\documentclass[12pt]{amsart}
\usepackage{amsfonts,latexsym,amscd}

\usepackage[all,cmtip,2cell]{xy}
\UseAllTwocells
\usepackage{verbatim}
\usepackage{hyperref}
\usepackage{mathrsfs}
\usepackage{amssymb}

\numberwithin{equation}{section}

\newtheorem{theorem}{Theorem}[section]
\newtheorem{lemma}[theorem]{Lemma}
\newtheorem{proposition}[theorem]{Proposition}
\newtheorem{corollary}[theorem]{Corollary}
\theoremstyle{definition}
\newtheorem{definition}[theorem]{Definition}

\newtheorem{example}[theorem]{Example}

\newtheorem{remark}[theorem]{Remark}

\newcommand{\id}{{\rm id}}

\newcommand{\End}{{\rm End}}
\newcommand{\Fun}{{\rm Fun}}

\newcommand{\Pro}{{\rm Pro}}

\newcommand{\Hom}{{\rm Hom}}

\newcommand{\Comod}{{\rm Comod}}

\newcommand{\Ind}{{\rm Ind}}

\newcommand{\Rep}{{\rm Rep}}
\newcommand{\sRep}{{\rm sRep}}

\newcommand{\Spec}{{\rm Spec}}

\newcommand{\Vect}{{\rm Vect}}

\newcommand{\g}{\mathfrak{g}}
\newcommand{\h}{\mathfrak{h}}

\newcommand{\surj}{\twoheadrightarrow}
\newcommand{\inj}{\hookrightarrow}

\newcommand{\ot}{\otimes}

\newcommand{\ben}{\begin{enumerate}}
\newcommand{\een}{\end{enumerate}}

\newcommand{\Lie}{{\rm Lie}}

\newcommand{\C}{{\mathcal C}}

\newcommand{\Mod}{\rm Mod}

\numberwithin{equation}{section}

\begin{document}

\title[Module
categories over affine supergroup schemes] {Module categories over affine supergroup schemes}

\author{Shlomo Gelaki}
\address{Department of Mathematics, Iowa State University, 
Ames IA, USA} \email{gelaki@iastate.edu}

\date{\today}

\keywords{affine supergroup scheme; tensor category;
module category; twist; triangular Hopf algebra}

\begin{abstract}
Let $k$ be an algebraically closed field of characteristic $0$ or $p>2$. Let $\mathcal{G}$ be an affine supergroup
scheme over $k$. We classify the indecomposable exact module categories over the tensor category ${\rm sCoh}_{\rm f}(\mathcal{G})$ of (coherent sheaves of) finite dimensional
$\mathcal{O}(\mathcal{G})$-supermodules in terms of $(\mathcal{H},\Psi)$-equivariant coherent sheaves on $\mathcal{G}$. We deduce from it the
classification of indecomposable {\em geometrical} module categories over $\sRep(\mathcal{G})$. When $\mathcal{G}$ is finite, this yields the
classification of {\em all} indecomposable exact module categories over the finite tensor category $\sRep(\mathcal{G})$. In particular, we obtain a
classification of twists for the supergroup algebra $k\mathcal{G}$ of a finite
supergroup scheme $\mathcal{G}$, and then combine it with \cite[Corollary 4.1]{EG3} to classify finite dimensional triangular Hopf algebras with the Chevalley property over $k$.
\end{abstract}

\maketitle

\tableofcontents

\section{Introduction}\label{Introduction}

Let $k$ be an algebraically closed field of characteristic $0$ or $p>2$.
Let $G$ be a finite group scheme over $k$. Consider the finite tensor category ${\rm Coh}(G)$ of finite dimensional $\mathcal{O}(G)$-modules over $k$, and the finite tensor category $\Rep(G)$
of finite dimensional rational representations of $G$ over $k$. In \cite{G} we classified the indecomposable exact
module categories over $\Rep(G)$, generalizing the classification of Etingof and Ostrik \cite{EO} for constant groups $G$. In particular, we obtained the
classification of twists for the group algebra $kG$, reproducing the classification
given by Movshev for constant groups $G$ in zero characteristic \cite{Mov}.

The goal of this paper is to extend \cite{G} to the \emph{super} case, and then combine it with \cite[Corollary 4.1]{EG3} to classify finite dimensional triangular Hopf algebras with the Chevalley property over $k$ (as promised in \cite[Remark 1.5(3)]{EG3}).

Let $\mathcal{G}$ be a finite supergroup scheme over $k$. Following \cite{G}, we first classify the indecomposable exact module categories over $\sRep(\mathcal{O}(\mathcal{G}))$, where $\mathcal{O}(\mathcal{G})$ is the coordinate Hopf superalgebra of $\mathcal{G}$, and then use the fact
that they are in bijection with the indecomposable exact module categories over $\sRep(\mathcal{G})$ \cite{EO} to get the classification of the
latter ones. The reason we approach it in this way is that $\sRep(\mathcal{O}(\mathcal{G}))$ is tensor equivalent to the tensor category
${\rm sCoh}_{\rm f}(\mathcal{G})=\text{sCoh}(\mathcal{G})$ (of coherent sheaves) of \emph{finite dimensional}
$\mathcal{O}(\mathcal{G})$-supermodules with the tensor product of convolution of sheaves, which allows us to use geometric tools and arguments. 

In fact, in Theorem \ref{grsch1} we
classify the indecomposable exact module categories over
${\rm sCoh}_{\rm f}(\mathcal{G})$, where $\mathcal{G}$ is \emph{any} affine supergroup scheme over $k$ (i.e., $\mathcal{G}$ is not necessarily finite). The classification is given in terms of certain $(\mathcal{H},\Psi)$-equivariant coherent sheaves on $\mathcal{G}$ (see Definition \ref{defequiv}). However when
$\mathcal{G}$ is {\em not} finite, not all indecomposable exact module categories
over $\sRep(\mathcal{G})$ are obtained from those over ${\rm sCoh}_{\rm f}(\mathcal{G})$ (see
Theorem \ref{modg} and Remark \ref{notall}); we refer to  those which are as {\em geometrical}. So the classification of exact module categories (even fiber functors) over $\sRep(\mathcal{G})$ for infinite affine supergroup schemes $\mathcal{G}$ remains unknown (even when $\mathcal{G}$ is a linear algebraic group over $\mathbb{C}$, see \cite{G}). 

As a consequence of our results, combined with \cite{AEGN, EO}, we obtain in Corollary \ref{twistsfinite} that gauge equivalence classes of twists
for the supergroup algebra $k\mathcal{G}$ of a \emph{finite} supergroup scheme $\mathcal{G}$ over $k$ are
parameterized by conjugacy classes of pairs $(\mathcal{H},\mathcal{J})$, where $\mathcal{H}\subset \mathcal{G}$ is a closed supergroup subscheme  and $\mathcal{J}$ is a \emph{non-degenerate} twist for $k\mathcal{H}$ (just as in the case of abstract finite groups). Furthermore, using Proposition \ref{quot} we show in Proposition \ref{minond} that a twist for $\mathcal{G}$ is non-degenerate if and only if it is \emph{minimal} (again, as for abstract finite groups). Finally, in Theorem \ref{classtrhas} we classify finite dimensional triangular Hopf algebras with the Chevalley property over $k$.

{\bf Acknowledgments.} The author is grateful to Pavel Etingof for stimulating and helpful discussions.

\section{Preliminaries}\label{prelim}

Throughout the paper we fix an algebraically closed field $k$ of characteristic $0$ or $p>2$. We refer the reader to the book \cite{egno} for the general theory of tensor categories. 

\subsection{Affine supergroup schemes}\label{gsch} 
We refer the reader to, e.g. \cite{W}, for preliminaries on affine group schemes over $k$, and to \cite{MA} for preliminaries on affine supergroup schemes over $k$.

Let $\mathcal{G}$ be an {\em affine supergroup scheme} over $k$, with unit morphism ${\rm e}:\Spec(k)\to \mathcal{G}$, inversion morphism
${\rm i}:\mathcal{G}\to \mathcal{G}$, and multiplication morphism ${\rm m}:\mathcal{G}\times \mathcal{G}\to \mathcal{G}$,
satisfying the usual group axioms. Recall that the coordinate algebra $\mathcal{O}(\mathcal{G})$\footnote{Some authors use $k[\mathcal{G}]$ instead.} of $\mathcal{G}$ is a  supercommutative Hopf superalgebra over $k$, and $\mathcal{G}$ is the functor from the category of supercommutative $k$-superalgebras to the category of groups defined by $R\mapsto \mathcal{G}(R):=\Hom_{{\rm SAlg}}(\mathcal{O}(\mathcal{G}),R)$ (so-called functor of points). Note that any affine supergroup scheme is the  inverse limit of affine supergroup schemes of {\em finite  type}.

A {\em closed} supergroup subscheme $\mathcal{H}$ of $\mathcal{G}$ is the spectrum of the Hopf quotient $\mathcal{O}(\mathcal{H}):=\mathcal{O}(\mathcal{G})/\mathcal{I}(\mathcal{H})$ by a Hopf ideal $\mathcal{I}(\mathcal{H})\subset \mathcal{O}(\mathcal{G})$. The ideal $\mathcal{I}(\mathcal{H})$ is referred to as the {\em defining ideal} of $\mathcal{H}$ in $\mathcal{O}(\mathcal{G})$.
For example, the {\em even part} of $\mathcal{G}$ is the closed group subscheme $\mathcal{G}_0\subset \mathcal{G}$ with the defining ideal $\mathcal{I}(\mathcal{G}_0)=\langle \mathcal{O}(\mathcal{G})_1 \rangle$, i.e., $\mathcal{G}_0$ is an ordinary affine group scheme with coordinate algebra $\mathcal{O}(\mathcal{G}_0)=\mathcal{O}(\mathcal{G})/\langle \mathcal{O}(\mathcal{G})_1 \rangle$. In particular, we have a surjective Hopf algebra map $
\pi:\mathcal{O}(\mathcal{G})\twoheadrightarrow \mathcal{O}(\mathcal{G}_0)$.

Let $\g=\g_0\oplus \g_1$ be the Lie superalgebra of $\mathcal{G}$, i.e., $\g$ is the space of left-invariant derivations of $\mathcal{O}(\mathcal{G})$, $\g_0$ is the space of even derivations of $\mathcal{O}(\mathcal{G})$, and $\g_1$ is the space of odd derivations of $\mathcal{O}(\mathcal{G})$. 
We have 
$\g=(\mathfrak{m}/ \mathfrak{m}^2)^*$, where $\mathfrak{m}\subset \mathcal{O}(\mathcal{G})$ is the kernel of the augmentation map, and $\g_0=\Lie(\mathcal{G}_0)$ is the Lie algebra of $\mathcal{G}_0$. 

Recall that $\mathcal{G}_0$ acts on $\g_1$ via the adjoint action. Let $a:\mathcal{G}_0\times\g_1^*\to \g_1^*$ be the coadjoint action of $\mathcal{G}_0$ on $\g_1^*$. Then 
$\wedge \g_1^*$ is an $\mathcal{O}(\mathcal{G}_0)$-comodule algebra with structure map 
$a^*:\wedge \g_1^*\to \mathcal{O}(\mathcal{G}_0)\ot \wedge\g_1^*$.

Since $\mathcal{O}(\mathcal{G}_0)$ is a quotient Hopf algebra of $\mathcal{O}(\mathcal{G})$, it follows that $\mathcal{O}(\mathcal{G})$ has a canonical structure of a left $\mathcal{O}(\mathcal{G}_0)$-comodule algebra with structure map $(\pi\ot\id)\Delta$. It is known \cite[Theorem 4.5]{MA} that the subalgebra of $\mathcal{O}(\mathcal{G}_0)$-coinvariants in $\mathcal{O}(\mathcal{G})$ is isomorphic to $\wedge \g_1^*$, and that we have a tensor decomposition
\begin{equation}\label{tensdecomp}
\mathcal{O}(\mathcal{G})\cong \wedge \g_1^*\otimes\mathcal{O}(\mathcal{G}_0)  
\end{equation}
of $\mathcal{O}(\mathcal{G}_0)$-supercomodule counital superalgebras. In particular, we have {\em abelian} equivalences
\begin{eqnarray}\label{btensdecomp}
\lefteqn{\text{sRep}(\mathcal{O}(\mathcal{G}))\cong\sRep(\wedge \g_1^*)\boxtimes_{\text{sVect}}\text{sRep}(\mathcal{O}(\mathcal{G}_0))}\\
& \cong & \sRep(\wedge \g_1^*)\boxtimes\text{Rep}(\mathcal{O}(\mathcal{G}_0))
\end{eqnarray}
such that $\text{Rep}(\mathcal{O}(\mathcal{G}_0))$ can be identified with a full tensor subcategory of $\text{sRep}(\mathcal{O}(\mathcal{G}))$ in the obvious way.

Recall that we have
$$\text{Rep}(\mathcal{O}(\mathcal{G}_0))=\text{Coh}_{\rm f}(\mathcal{G}_0)=\bigoplus_{g\in \mathcal{G}_0(k)}\text{Coh}_{\rm f}(\mathcal{G}_0)_{g},$$
where $\text{Coh}_{\rm f}(\mathcal{G}_0)_{g}$ is the abelian subcategory of sheaves supported at $g$, with unique simple 
object $\delta_{g}$ and indecomposable projective object $P_{g}:=\widehat{\mathcal{O}(\mathcal{G}_0)_{g}}$ in the pro-completion category, where $\mathcal{O}(\mathcal{G}_0)_{g}$ is the completion of $\mathcal{O}(\mathcal{G}_0)$ at $g$ \cite[Section 3.1]{G}. Thus by (\ref{btensdecomp}), we have 
\begin{equation}\label{btensdecomp1}
\text{sRep}(\mathcal{O}(\mathcal{G}))\cong\bigoplus_{g\in \mathcal{G}_0(k)}\sRep(\wedge \g_1^*)\boxtimes\text{Coh}_{\rm f}(\mathcal{G}_0)_{g}
\end{equation}
as {\em abelian} categories. 

Recall that closed supergroup subschemes $\mathcal{H}\subset\mathcal{G}$ are in bijection with pairs $(\mathcal{H}_0,\h_1)$, where $\mathcal{H}_0\subset \mathcal{G}_0$ is a closed group subscheme, $\h_1\subset \g_1$ is an $\mathcal{H}_0$-invariant subspace, and $[\h_1,\h_1]\subset \h_0:=\Lie(\mathcal{H}_0)$ (see, e.g., \cite[Section 6.2]{MS}). 

Let $\Psi: \mathcal{G}\times \mathcal{G}\to \mathbb{G}_m$ be a normalized even $2$-cocycle.
Equivalently, $\Psi\in \mathcal{O}(\mathcal{G})\ot \mathcal{O}(\mathcal{G})$ is a twist for $\mathcal{O}(\mathcal{G})$, i.e., $\Psi$ is an  
invertible even element satisfying the equations
$$(\Delta\ot \id)(\Psi)(\Psi\ot 1)=(\id\ot \Delta)(\Psi)(1\ot \Psi),$$
$$(\varepsilon\ot \id)(\Psi)=(\id\ot \varepsilon)(\Psi)=1.$$

Finally, recall that a {\em finite} supergroup scheme $\mathcal{G}$ is an affine supergroup scheme whose function algebra $\mathcal{O}(\mathcal{G})$ is finite dimensional. In this case, $k\mathcal{G}:=\mathcal{O}(\mathcal{G})^*$ is a supercocommutative Hopf superalgebra (called the {\em group  algebra} of $\mathcal{G}$).

\subsection{Module categories over tensor categories.}\label{Module categories over tensor categories}
Let $\mathcal{C}$ be a tensor category over $k$.
Let $\Ind(\mathcal{C})$ and $\Pro(\mathcal{C})$ be the categories of
$\Ind$-objects and $\Pro$-objects of $\mathcal{C}$, respectively. It is well known that the tensor
structure on $\mathcal{C}$ extends to a tensor structure on
$\Ind(\mathcal{C})$ and $\Pro(\mathcal{C})$. However $\Ind(\mathcal{C})$ and $\Pro(\mathcal{C})$ are not rigid, but the rigid structure on $\mathcal{C}$ induces two duality functors $\Pro(\mathcal{C})\to\Ind(\mathcal{C})$ (``continuous dual") and $\Ind(\mathcal{C})\to \Pro(\mathcal{C})$ (``linear dual"), which we shall both denote by $X\mapsto X^*$; they are antiequivalence inverses of each other. It is also known that $\Ind(\mathcal{C})$ has enough injectives. 

Recall that a (left) \emph{module category} $\mathcal{M}$ over $\mathcal{C}$ is a locally finite abelian category equipped with a (left) action $\ot ^{\mathcal{M}}:\mathcal{C}\boxtimes \mathcal{M}\to \mathcal{M}$, such that the bifunctor $\ot ^{\mathcal{M}}$ is bilinear on morphisms and biexact. Recall also that $\mathcal{M}$ is \emph{exact} if any additive module functor $\mathcal{M}\to \mathcal{M}_1$ from $\mathcal{M}$ to any other
$\mathcal{C}$-module category $\mathcal{M}_1$ is exact, and that $\mathcal{M}$ is {\em indecomposable} if $\mathcal{M}$ is not equivalent to a direct sum of two nontrivial module subcategories. It is also known that the $\mathcal{C}$-module structure on $\mathcal{M}$ extends to a module structure on $\Ind(\mathcal{M})$ over $\Ind(\mathcal{C})$. Moreover, $\mathcal{M}$ is exact if and only if 
for any $M\in \mathcal{M}$
and any injective object $I\in \Ind(\mathcal{C})$ (resp., projective object $P\in \Pro(\mathcal{C})$), $I\ot M$ is injective in $\Ind(\mathcal{M})$ (resp., $P\ot M$ is projective in $\Pro(\mathcal{M})$) (see \cite[Propositions 3.11, 3.16]{EO}, \cite[Proposition 2.4]{G}).

Following \cite{EO}, we say that two simple objects $M_1,M_2\in \mathcal{M}$ are \textit{related} if there exists an object $X\in \mathcal{C}$ such that
$M_1$ appears as a subquotient in $X\ot ^{\mathcal{M}}M_2$. This defines an equivalence relation, and $\mathcal{M}$
decomposes into a direct sum $\mathcal{M}=\oplus \mathcal{M}_i$ of indecomposable exact module subcategories indexed by the equivalence classes (see \cite[Lemma 3.8 \&
Proposition 3.9]{EO} and \cite[Proposition 2.5]{G}).

Assume $\mathcal{M}$ is exact. Recall that an object $\delta\in\mathcal{M}$ {\em generates} $\mathcal{M}$ if for
any $M\in \mathcal{M}$ there exists $X\in\mathcal{C}$ such that
$\Hom_{\mathcal{M}}(X\ot ^{\mathcal{M}}\delta,M)\ne 0$.
It is known that $\delta$ generates $\mathcal{M}$ if and only if for any $M\in \mathcal{M}$ there exists $X\in\mathcal{C}$ such that
$M$ is a subquotient of $X\ot ^{\mathcal{M}}\delta$
(cf. \cite{EO}). Thus if 
$\mathcal{M}$ is indecomposable and $\delta$
is simple, then $\delta\in \mathcal{M}$ generates $\mathcal{M}$.

Finally recall that for every two objects $M_1,M_2\in \mathcal{M}$, we have an object $\overline{\Hom}(M_1,M_2)\in \Pro(\mathcal{C})$ satisfying 
$$\Hom_{\mathcal{M}}(M_2,X\ot^{\mathcal{M}} M_1)
\cong\Hom_{\text{Pro}(\mathcal{C})} (\overline{\Hom}(M_1,M_2),X),\,\,X\in \mathcal{C}$$
(the \emph{dual internal \Hom}). For every $M\in \mathcal{M}$, the pro-object
$\overline{\Hom}(M,M)$ has a canonical
structure of a coalgebra. In terms of
internal Hom's \cite{EO}, the algebra $\underline{\Hom}(M,M)$
in $\Ind(\mathcal{C})$ is isomorphic to the dual algebra 
$(\overline{\Hom}(M,M))^*$ under the duality functor $^*:\Pro(\mathcal{C})\to\Ind(\mathcal{C})$. Now if $\mathcal{M}$ is indecomposable and exact, we have a $\mathcal{C}$-module equivalence 
$\mathcal{M}\cong\Comod_{\Pro(\mathcal{C})}(\overline{\Hom}(M,M))$.

\section{The tensor category $\text{sCoh}_{\rm f}(\mathcal{G})$}\label{The tensor category sCohfG}

Let $\mathcal{G}$ be an affine supergroup scheme\footnote{The purely even case is treated in \cite[Section 3.1]{G}.} over $k$, 
and let $$\overline{\mathcal{O}(\mathcal{G})}=\mathcal{O}(\mathcal{G})\times k\langle u \rangle$$ be the Radford's biproduct ordinary Hopf algebra, where $u$ is a grouplike element of order $2$ acting on $\mathcal{O}(\mathcal{G})$ by parity, and
\begin{equation}\label{Radford biproduct ordinary Hopf algebra}
\overline{\Delta}(x)=\sum (x_1\ot u^{|x_2|})\ot (x_2\ot 1)
\end{equation}
for every homogeneous element $x\in \mathcal{O}(\mathcal{G})$, where $\Delta(x)=\sum x_1\ot x_2$. 
Recall that we have an equivalence of tensor categories 
\begin{equation}\label{we have an equivalence of tensor categories}
\Rep(\overline{\mathcal{O}(\mathcal{G})})\cong
\text{sRep}(\mathcal{O}(\mathcal{G})).
\end{equation}
In particular, $\text{Rep}(\mathcal{O}(\mathcal{G}_0))$ is a tensor subcategory of $\Rep(\overline{\mathcal{O}(\mathcal{G})})$.

\begin{definition}
Let $\text{sCoh}_{\rm f}(\mathcal{G})$ (resp., $\text{sQCoh}(\mathcal{G})$) be the tensor category (resp., monoidal category) of finite dimensional (resp., all) representations of the Hopf algebra $\overline{\mathcal{O}(\mathcal{G})}$.
\end{definition}
By definition, we have equivalences of tensor and monoidal categories 
$$\text{sCoh}_{\rm f}(\mathcal{G})\cong\sRep(\mathcal{O}(\mathcal{G}))\,\,\,
\text{and}\,\,\,\text{sQCoh}(\mathcal{G})\cong{\rm SRep}(\mathcal{O}(\mathcal{G})),$$
respectively, where $\sRep(\mathcal{O}(\mathcal{G}))$ and ${\rm SRep}(\mathcal{O}(\mathcal{G}))$ are the categories of \emph{finite dimensional} and \emph{all} representations of the Hopf superalgebra $\mathcal{O}(\mathcal{G})$ on $k$-supervector spaces, respectively. 

We have that $\Ind(\text{sCoh}_{\rm f}(\mathcal{G}))$ is the category of {\em locally finite} quasi-coherent sheaves of $\mathcal{O}(\mathcal{G})$-supermodules (i.e., representations in which every vector generates a finite dimensional subrepresentation). 

\begin{remark}\label{think}
By a {\em quasi-coherent sheaf on $\mathcal{G}$} we will mean a quasi-coherent sheaf  of $\mathcal{O}(\mathcal{G})$-\emph{supermodules}, and by a {\em finite} quasi-coherent sheaf on $\mathcal{G}$ we will mean a quasi-coherent sheaf  of {\em finite dimensional} $\mathcal{O}(\mathcal{G})$-\emph{supermodules}. Note that finite quasi-coherent sheaves on $\mathcal{G}$ are automatically supported on finite sets in $\mathcal{G}_0$. Thus, one can think of $\text{sCoh}_{\rm f}(\mathcal{G})$ and $\text{sQCoh}(\mathcal{G})$ as the $k$-linear abelian categories of finite quasi-coherent sheaves and quasi-coherent sheaves on $\mathcal{G}$, respectively (which explains our notation). In particular, the tensor products in 
$\text{sCoh}_{\rm f}(\mathcal{G})$ and $\text{sQCoh}(\mathcal{G})$ correspond to the convolution product of sheaves
\begin{equation}\label{m_*}
\text{X}\otimes \text{Y}:={\rm m}_*(\text{X}\boxtimes \text{Y})
\end{equation}
(where ${\rm m}_*$ is the direct image functor of ${\rm m}$). Notice that the tensor category $\text{Coh}_{\rm f}(\mathcal{G}_0)$ is identified with the tensor subcategory of $\text{sCoh}_{\rm f}(\mathcal{G})$ consisting of sheaves on which odd elements act trivially.
\end{remark}

We will also consider the following categories.

\begin{definition}
Let $\text{Coh}_{\rm f}(\mathcal{G})$ (resp., $\text{QCoh}(\mathcal{G})$) be the abelian category of finite dimensional (resp., all) representations of the algebra $\mathcal{O}(\mathcal{G})$.
\end{definition}

Note that $\text{Coh}_{\rm f}(\mathcal{G})$ is not a tensor category when $\mathcal{G}$ is not even, and that we have a tensor equivalence 
$\text{sCoh}_{\rm f}(\mathcal{G}_0)\cong\text{Coh}_{\rm f}(\mathcal{G}_0)\boxtimes {\rm sVect}$.
However, we do have the following.

\begin{lemma}\label{modcatstr}
The abelian category ${\rm Coh}_{\rm f}(\mathcal{G})$ has a natural structure of a left module category over ${\rm sCoh}_{\rm f}(\mathcal{G})$, given by
$$
{\rm sCoh}_{\rm f}(\mathcal{G})\boxtimes {\rm Coh}_{\rm f}(\mathcal{G})\to {\rm Coh}_{\rm f}(\mathcal{G}),\,\,\,X\boxtimes Y\mapsto 
{\rm m}_*(X\boxtimes Y).
$$
Moreover, the restriction functor
${\rm sCoh}_{\rm f}(\mathcal{G})\twoheadrightarrow {\rm Coh}_{\rm f}(\mathcal{G})$, 
induced by the algebra inclusion $\mathcal{O}(\mathcal{G})\subset \overline{\mathcal{O}(\mathcal{G})}$, 
has a canonical structure of a surjective ${\rm sCoh}_{\rm f}(\mathcal{G})$-module functor.
\end{lemma}

\begin{proof}
Since by (\ref{Radford biproduct ordinary Hopf algebra}), $\mathcal{O}(\mathcal{G})\subset \overline{\mathcal{O}(\mathcal{G})}$ is a left coideal subalgebra, the claim follows from (\ref{we have an equivalence of tensor categories}).
\end{proof}

For every $g\in \mathcal{G}_0(k)$, let $\text{sCoh}_{\rm f}(\mathcal{G})_{g}:=\sRep(\wedge \g_1^*)\boxtimes\text{Coh}_{\rm f}(\mathcal{G}_0)_{g}$. By (\ref{btensdecomp}), we have an {\em abelian} equivalence
\begin{equation}\label{scohf}
\text{sCoh}_{\rm f}(\mathcal{G})\cong \bigoplus_{g\in \mathcal{G}_0(k)}\text{sCoh}_{\rm f}(\mathcal{G})_{g}.
\end{equation}

We will need the following result.

\begin{lemma}\label{tensubcat}
Every tensor subcategory of ${\rm sCoh}_{\rm f}(\mathcal{G})$ is either of the form ${\rm sCoh}_{\rm f}(\mathcal{H})$ or ${\rm Coh}_{\rm f}(\mathcal{H})$ for some closed supergroup subscheme $\mathcal{H}\subset\mathcal{G}$ or closed subgroup scheme $\mathcal{H}\subset\mathcal{G}_0$, respectively. 
\end{lemma}

\begin{proof}
It is known that every tensor subcategory of $\text{Rep}(\overline{\mathcal{O}(\mathcal{G})})$ corresponds to a Hopf quotient of $\overline{\mathcal{O}(\mathcal{G})}$. Now if $u$ is mapped to $1$ in the quotient, then we get the second case (as all odd elements must act by zero). Otherwise, we get the first case.
\end{proof}

\begin{remark}\label{scohfgomega}
The class of tensor categories ${\rm sCoh}_{\rm f}(\mathcal{G})$ can be extended to a larger class of tensor categories ${\rm sCoh}_{\rm f}(\mathcal{G},\Omega)$ in exactly the same way as in the even case \cite[Section 5]{G}. Namely, let $\mathcal{G}$ be an affine supergroup scheme over $k$, and let $\Omega\in Z^3(\mathcal{G},\mathbb{G}_m)$ be a normalized even $3$-cocycle. Equivalently, $\Omega$ is a 
\emph{Drinfeld associator} for $\mathcal{O}(\mathcal{G})$, i.e., $\Omega\in \mathcal{O}(\mathcal{G})^{\ot 3}$ is an 
invertible even element satisfying the equations
$$(\id\ot \id\ot \Delta)(\Omega)(\Delta\ot \id\ot \id)(\Omega)=(1\ot \Omega)(\id\ot \Delta\ot \id)(\Omega)(\Omega\ot 1)$$
and
$$(\varepsilon\ot \id\ot \id)(\Omega)=(\id\ot \varepsilon\ot \id)(\Omega)=(\id\ot \id\ot \varepsilon)(\Omega)=1.$$
Then ${\rm sCoh}_{\rm f}(\mathcal{G},\Omega)$ is the abelian category ${\rm sCoh}_{\rm f}(\mathcal{G})$ equipped with the tensor product given by convolution and associativity constraint given by the action of $\Omega$ (viewed as an invertible element in $\mathcal{O}(\mathcal{G})^{\ot 3}$).
\end{remark}

\section{Equivariant quasi-coherent sheaves}\label{equiqcohsh}

Let $\mathcal{G}$ be an affine supergroup scheme\footnote{The purely even case is treated in \cite[Section 3.2]{G}.} over $k$, let $\mathcal{H}\subset \mathcal{G}$ be a closed supergroup subscheme (see \ref{gsch}), and let
$\iota=\iota_{\mathcal{H}}:\mathcal{H}\hookrightarrow \mathcal{G}$ be the inclusion morphism. Let $\mu
:\mathcal{G}\times \mathcal{H}\to \mathcal{G}$ be the free action of $\mathcal{H}$ on $\mathcal{G}$ by right translations (in other words, the free actions of $\mathcal{H}(R)$ on $\mathcal{G}(R)$ by right translations that are functorial in $R$, $R$ a supercommutative $k$-superalgebra). Set
$$\eta:=\mu(\id\times {\rm m})=\mu(\mu\times \id):
\mathcal{G}\times \mathcal{H}\times \mathcal{H}\to \mathcal{G},$$ and let
$$\text{p}_1:\mathcal{G}\times \mathcal{H}\to \mathcal{G},\,\text{p}_1:\mathcal{G}\times \mathcal{H}\times \mathcal{H}\to \mathcal{G},\,
\text{p}_{12}:\mathcal{G}\times \mathcal{H}\times \mathcal{H}\to \mathcal{G}\times \mathcal{H}$$ be the obvious projections. We clearly have $\text{p}_1\circ
\text{p}_{12}=\text{p}_1$ as morphisms $\mathcal{G}\times \mathcal{H}\times \mathcal{H}\to \mathcal{G}$.

Now let $\Psi: \mathcal{H}\times \mathcal{H}\to \mathbb{G}_m$ be a normalized even $2$-cocycle, i.e., 
$\Psi\in \mathcal{O}(\mathcal{H})^{\ot 2}$ is a twist for $\mathcal{O}(\mathcal{H})$ (see \ref{gsch}), and let
$\mathcal{O}(\mathcal{H})_{\Psi}$ be the (``twisted") supercoalgebra with underlying supervector
space $\mathcal{O}(\mathcal{H})$ and comultiplication $\Delta_{\Psi}$ given by
$\Delta_{\Psi}(f):=\Delta(f)\Psi$, where $\Delta$ is the standard comultiplication of $\mathcal{O}(\mathcal{H})$. Note that $\Psi$ defines an automorphism of any quasi-coherent sheaf on $\mathcal{H}\times \mathcal{H}$ by multiplication.

\begin{definition}\label{defequiv}
Let $\Psi: \mathcal{H}\times \mathcal{H}\to \mathbb{G}_m$ be a normalized even $2$-cocycle on a closed supergroup subscheme  $\mathcal{H}\subset \mathcal{G}$.
\begin{enumerate}
\item 
An \emph{$(\mathcal{H},\Psi)$-equivariant} quasi-coherent sheaf  on $\mathcal{G}$ is a pair
$(S,\lambda)$, where $S\in \text{sQCoh}(\mathcal{G})$ and $\lambda:{\rm p}_1^*(S)\xrightarrow{\cong} \mu^*(S)$ is an
isomorphism of sheaves on $\mathcal{G}\times \mathcal{H}$, such that the diagram of morphisms of sheaves on
$\mathcal{G}\times \mathcal{H}\times \mathcal{H}$
\begin{equation*}
\label{equivariantX} \xymatrix{{\rm p}_1^*(S) \ar[d]_{(\id\times
{\rm m})^*(\lambda)}\ar[rr]^{{\rm p}_{12}^*(\lambda)}
&& (\mu\circ {\rm p}_{12})^*(S)\ar[d]^{(\mu\times \id)^*(\lambda)} &&\\
\eta^*(S)\ar[rr]_{\id\boxtimes \Psi} && \eta^*(S)&&}
\end{equation*}
is commutative.
\item
Let $(S,\lambda_S)$ and $(T,\lambda_T)$ be two
$(\mathcal{H},\Psi)$-equivariant quasi-coherent sheaves on $\mathcal{G}$. A morphism $\phi:S\to T$ in
$\text{sQCoh}(\mathcal{G})$ is said to be \emph{$(\mathcal{H},\Psi)$-equivariant} if the
diagram of morphisms of sheaves on $\mathcal{G}\times \mathcal{H}$
\begin{equation*}
\label{equivariantX1} \xymatrix{{\rm p}_1^*(S)
\ar[d]_{\lambda_S}\ar[rr]^{{\rm p}_1^*(\Psi)}
&& {\rm p}_1^*(T)\ar[d]^{\lambda_T} &&\\
\mu^*(S)\ar[rr]_{\mu^*(\phi)} && \mu^*(T)&&}
\end{equation*}
is commutative.
\item
Let $\text{sCoh}_{\rm f}^{(\mathcal{H},\Psi)}(\mathcal{G})$ be the $k$-abelian category of $(\mathcal{H},\Psi)$-equivariant {\em coherent} sheaves on $\mathcal{G}$ with \emph{finite support in $\mathcal{G}_0/\mathcal{H}_0$} (i.e., sheaves supported on {\em finitely many} $\mathcal{H}_0$-cosets), with $(\mathcal{H},\Psi)$-equivariant morphisms.
\end{enumerate}
\end{definition}

Replacing $\text{sQCoh}(\mathcal{G})$ with $\text{QCoh}(\mathcal{G})$ everywhere in Definition \ref{defequiv}, we define the notion of an {\em $(\mathcal{H},\Psi)$-equivariant} $\mathcal{O}(\mathcal{G})$-module, and the $k$-abelian category $\text{Coh}_{\rm f}^{(\mathcal{H},\Psi)}(\mathcal{G})$ of {\em finitely generated} $(\mathcal{H},\Psi)$-equivariant $\mathcal{O}(\mathcal{G})$-modules with finite support in $\mathcal{G}_0/\mathcal{H}_0$.

\begin{example}\label{ex1}
We have 
$$\text{sCoh}_{\rm f}^{(\{1\},1)}(\mathcal{G})=\text{sCoh}_{\rm f}(\mathcal{G})\,\,\,{\rm and}\,\,\,\text{Coh}_{\rm f}^{(\{1\},1)}(\mathcal{G})=\text{Coh}_{\rm f}(\mathcal{G}).$$
\end{example}

\begin{remark}\label{biequiv}
Let $(\mathcal{H}',\Psi')$ be another pair consisting of a closed supergroup subscheme  $\mathcal{H}'\subset\mathcal{G}$ and an even normalized $2$-cocycle $\Psi'$ on $\mathcal{H}'$. By considering the free right action of $\mathcal{H}'\times \mathcal{H}$ on $\mathcal{G}$ given by $g(a,b):=a^{-1}gb$, we can similarly
define {\em $((\mathcal{H}',\Psi'),(\mathcal{H},\Psi))$-biequivariant} quasi-coherent sheaves on $\mathcal{G}$, $((\mathcal{H}',\Psi'),(\mathcal{H},\Psi))$-equivariant $\mathcal{O}(\mathcal{G})$-modules, and the $k$-abelian categories $\text{sCoh}_{\rm f}^{((\mathcal{H}',\Psi'),(\mathcal{H},\Psi))}(\mathcal{G})$, $\text{Coh}_{\rm f}^{((\mathcal{H}',\Psi'),(\mathcal{H},\Psi))}(\mathcal{G})$.
\end{remark}

\begin{remark}\label{scohfgomega1}
Retain the notation from Remark \ref{scohfgomega}. 
Let $\mathcal{H}\subset \mathcal{G}$ be a closed supergroup subscheme, and let $\Psi\in C^2(\mathcal{H},\mathbb{G}_m)$ be a normalized even $2$-cochain such that $d\Psi=\Omega_{|\mathcal{H}}$. Then similarly to ${\rm sCoh}_{\rm f}^{(\mathcal{H},\Psi)}(\mathcal{G})$ (the case $\Omega=1$), with the obvious adjustments, we can define the category ${\rm sCoh}_{\rm f}^{(\mathcal{H},\Psi)}(\mathcal{G},\Omega)$ of $(\mathcal{H},\Psi)$-equivariant coherent sheaves on $(\mathcal{G},\Omega)$ with \emph{finite} support in $\mathcal{G}_0/\mathcal{H}_0$, and the category ${\rm Coh}_{\rm f}^{(\mathcal{H},\Psi)}(\mathcal{G},\Omega)$ of {\em finitely generated} $(\mathcal{H},\Psi)$-equivariant $\mathcal{O}(\mathcal{G},\Omega)$-modules, where $\mathcal{O}(\mathcal{G},\Omega)$ is the obviously defined quasi-Hopf algebra.
\end{remark}

Consider now the supercoalgebra $\mathcal{O}(\mathcal{H})_{\Psi}$ in $\text{sCoh}(\mathcal{G})$, and let $\widehat{\mathcal{O}(\mathcal{H})_{\Psi}}$ be its profinite completion with respect to the superalgebra structure of $\mathcal{O}(\mathcal{H})$ (see \cite[Example 2.4]{G}). Then $\widehat{\mathcal{O}(\mathcal{H})_{\Psi}}$ is a supercoalgebra object in both $\Pro({\rm sCoh}(\mathcal{G}))$ and $\Pro({\rm sCoh}_{\rm f}(\mathcal{G}))$.

\begin{lemma}\label{eqcscom}
We have abelian equivalences $${\rm sCoh}_{\rm f}^{(\mathcal{H},\Psi)}(\mathcal{G})\cong{\rm Comod}_{\Pro({\rm sCoh}_{\rm f}(\mathcal{G}))}(\widehat{\mathcal{O}(\mathcal{H})_{\Psi}})$$ and 
$${\rm Coh}_{\rm f}^{(\mathcal{H},\Psi)}(\mathcal{G})\cong{\rm Comod}_{\Pro({\rm Coh}_{\rm f}(\mathcal{G}))}(\widehat{\mathcal{O}(\mathcal{H})_{\Psi}}).$$
\end{lemma}

\begin{proof}
We prove the first equivalence, the proof of the second one being similar.

For every $S\in \Pro({\rm sCoh}_{\rm f}(\mathcal{G}))$, we have a natural isomorphism $$\Hom_{\mathcal{G}\times
\mathcal{H}}(\mu^*(S),{\rm p}_1^*(S))\cong \Hom_{\mathcal{G}}(S,\mu_*{\rm p}_1^*(S))$$ (``adjunction"). Since $\mu_*{\rm p}_1^*(S)\cong S\ot
\widehat{\mathcal{O}(\mathcal{H})}$, we can assign
to any isomorphism $\lambda:\mu^*(S)\to {\rm p}_1^*(S)$ a morphism
$\rho_{\lambda}:S\to S\ot \widehat{\mathcal{O}(\mathcal{H})}$. It is now straightforward
to verify that $\rho_{\lambda}:S\to S\ot
\widehat{\mathcal{O}(\mathcal{H})_{\Psi}}$ is a comodule map 
 if and only if $(S,\lambda^{-1})$ is an $(\mathcal{H},\Psi)$-equivariant coherent sheaf on $\mathcal{G}$ with finite support in $\mathcal{G}_0/\mathcal{H}_0$.
\end{proof}

The next proposition will be very useful in the sequel.

\begin{proposition}\label{simpleex}
Let $\mathcal{H}\subset \mathcal{G}$ be a closed supergroup subscheme, and let $\Psi$ be an even normalized $2$-cocycle on $\mathcal{H}$.
Then the following hold:
\begin{enumerate}
\item
The structure sheaf $\mathcal{O}(\mathcal{H})$ of $\mathcal{H}$ admits a canonical structure of an $(\mathcal{H},\Psi)$-equivariant coherent sheaf on $\mathcal{H}$, making it a simple object
of ${\rm sCoh}_{\rm f}^{(\mathcal{H},\Psi)}(\mathcal{H})\cong{\rm sVect}$, and the regular $\mathcal{O}(\mathcal{H})$-module admits a canonical structure of an $(\mathcal{H},\Psi)$-equivariant $\mathcal{O}(\mathcal{H})$-module, making it the simple object of ${\rm Coh}_{\rm f}^{(\mathcal{H},\Psi)}(\mathcal{H})\cong{\rm Vect}$.
\item 
The sheaf $\iota_*\mathcal{O}(\mathcal{H})$\footnote{The superrepresentation of $\mathcal{O}(\mathcal{G})$ on $\mathcal{O}(\mathcal{H})$ coming from $\iota$.} is a simple object in
${\rm sCoh}_{\rm f}^{(\mathcal{H},\Psi)}(\mathcal{G})$, and the $\mathcal{O}(\mathcal{G})$-module $\iota_*\mathcal{O}(\mathcal{H})\in {\rm Coh}(\mathcal{G})$\footnote{The representation of $\mathcal{O}(\mathcal{G})$ on $\mathcal{O}(\mathcal{H})$ coming from $\iota$.} is a simple object in
${\rm Coh}_{\rm f}^{(\mathcal{H},\Psi)}(\mathcal{G})$.
\item
For every $X\in {\rm sCoh}_{\rm f}(\mathcal{G})$, we have 
$${\rm m}_*(X\boxtimes M)\in {\rm sCoh}_{\rm f}^{(\mathcal{H},\Psi)}(\mathcal{G}),\,\,M\in {\rm sCoh}_{\rm f}^{(\mathcal{H},\Psi)}(\mathcal{G}),$$
and
$${\rm m}_*(X\boxtimes M)\in {\rm Coh}_{\rm f}^{(\mathcal{H},\Psi)}(\mathcal{G}),\,\,M\in {\rm Coh}_{\rm f}^{(\mathcal{H},\Psi)}(\mathcal{G}).$$
\end{enumerate}
\end{proposition}

\begin{proof}
We will prove the proposition for sheaves, the proof for modules being similar.

(1) Consider the isomorphism $\varphi:=({\rm m},{\rm p}_2):\mathcal{H}\times
\mathcal{H}\xrightarrow{\cong} \mathcal{H}\times \mathcal{H}
$. Since $\rm{p}_1\circ\varphi={\rm m}$, it follows that $({\rm p}_1\circ\varphi)^*\mathcal{O}(\mathcal{H})={\rm m}^*\mathcal{O}(\mathcal{H})$. Now,
multiplication by $\Psi$ defines an isomorphism
$${\rm m}^*\mathcal{O}(\mathcal{H})=({\rm p}_1\circ\varphi)^*\mathcal{O}(\mathcal{H})=(\varphi^*\circ {\rm p}_1^*)\mathcal{O}(\mathcal{H})\xrightarrow{\Psi}
(\varphi^*\circ {\rm p}_1^*)\mathcal{O}(\mathcal{H}),$$ and since we have 
${\rm p}_1^*\mathcal{O}(\mathcal{H})=\mathcal{O}(\mathcal{H})\ot \mathcal{O}(\mathcal{H})$, we get an
isomorphism $$\lambda:{\rm p}_1^*\mathcal{O}(\mathcal{H})=
\mathcal{O}(\mathcal{H})\ot \mathcal{O}(\mathcal{H})\xrightarrow{\varphi^*}
\varphi^*(\mathcal{O}(\mathcal{H})\ot \mathcal{O}(\mathcal{H}))\xrightarrow{\Psi^{-1}}
{\rm m}^*\mathcal{O}(\mathcal{H}).$$ The fact that
$(\mathcal{O}(\mathcal{H}),\lambda)$ is an $(\mathcal{H},\Psi)$-equivariant coherent sheaf on $\mathcal{H}$ can be checked now in a straightforward manner using the tensor decomposition (\ref{tensdecomp}). Clearly, $(\mathcal{O}(\mathcal{H}),\lambda)$ is a simple object in ${\rm sCoh}_{\rm f}^{(\mathcal{H},\Psi)}(\mathcal{H})$.

Let $\delta:=(\mathcal{O}(\mathcal{H}),\lambda)$, and consider the simple object $\delta^-:=k^{0\mid 1}\ot \delta$ (via $\id\ot\Delta$). 
It is clear that $\delta\ncong\delta^-$ in ${\rm sCoh}_{\rm f}^{(\mathcal{H},\Psi)}(\mathcal{H})$. 

Now let $M$ be any object in ${\rm sCoh}_{\rm f}^{(\mathcal{H},\Psi)}(\mathcal{H})$, and let $X:=M^{{\rm co}\mathcal{O}(\mathcal{H})}$. We claim that 
$$
M\cong M^{{\rm co}\mathcal{O}(\mathcal{H})}\ot_k\delta:=\overline{X}_0\ot_k\delta\oplus \overline{X}_1\ot_k\delta^-
$$
in ${\rm sCoh}_{\rm f}^{(\mathcal{H},\Psi)}(\mathcal{H})$, where $\overline{X}$ denotes the underlying vector space of $X$.
Indeed, let $\alpha: M^{{\rm co}\mathcal{O}(\mathcal{H})}\ot_k \mathcal{O}(\mathcal{H})\to M$ be the action map, and let 
$$\beta:M\to M^{{\rm co}\mathcal{O}(\mathcal{H})}\ot_k\mathcal{O}(\mathcal{H}),\,\,m\mapsto \sum S^{-1}(m_1)\cdot m_0\ot m_2.$$
Then it is straightforward to check that $\alpha$ and $\beta$ are inverse to each other. 
Hence, ${\rm sCoh}_{\rm f}^{(\mathcal{H},\Psi)}(\mathcal{H})$ is semisimple of rank $2$, as claimed.

(2) Since $\iota$ is affine, the commutative diagrams
\begin{equation*}
\xymatrix{\mathcal{H}\times \mathcal{H} \ar[d]_{\iota \times \id}\ar[rr]^{{\rm p}_1}
&& \mathcal{H} \ar[d]^{\iota} &&\\
\mathcal{G}\times \mathcal{H}\ar[rr]_{{\rm p}_1} && \mathcal{G}&&} \xymatrix{\mathcal{H}\times \mathcal{H} \ar[d]_{\iota
\times \id}\ar[rr]^{{\rm m}}
&& \mathcal{H} \ar[d]^{\iota} &&\\
\mathcal{G}\times \mathcal{H}\ar[rr]_{\mu} && \mathcal{G}&&}
\end{equation*}
yield isomorphisms 
\begin{equation}\label{eq1}
{\rm p}_1^*\iota_*\mathcal{O}(\mathcal{H})\xrightarrow{\cong}
(\iota \times \id)_*{\rm p}_1^* \mathcal{O}(\mathcal{H})
\end{equation} 
and
\begin{equation}\label{eq2}
(\iota \times \id)_*{\rm m}^* \mathcal{O}(\mathcal{H}) \xrightarrow{\cong}
\mu^*\iota_*\mathcal{O}(\mathcal{H})
\end{equation} 
(``base change").

Let $\lambda:{\rm p}_1^*\mathcal{O}(\mathcal{H})\xrightarrow{\cong}
{\rm m}^*\mathcal{O}(\mathcal{H})$ be the isomorphism constructed in Part (1). Since $\iota$ is $\mathcal{H}$-equivariant, we get an isomorphism
\begin{equation}\label{eq3}
(\iota \times
\id)_*{\rm p}_1^* \mathcal{O}(\mathcal{H}) \xrightarrow{(\iota \times
\id)_*(\lambda)} (\iota \times \id)_*{\rm m}^* \mathcal{O}(\mathcal{H}).
\end{equation}
It is now straightforward to check, using the tensor decomposition (\ref{tensdecomp}), that the composition of isomorphisms (\ref{eq1}), (\ref{eq2}) and (\ref{eq3})
\begin{equation*}
{\rm p}_1^*\iota_*\mathcal{O}(\mathcal{H}) \xrightarrow{\cong} 
\mu^*\iota_*\mathcal{O}(\mathcal{H})
\end{equation*}
endows $\iota_*\mathcal{O}(\mathcal{H})$ with a structure of an $(\mathcal{H},\Psi)$-equivariant coherent sheaf on $\mathcal{G}$. Clearly, $\iota_*\mathcal{O}(\mathcal{H})$ is simple.

(3) Consider the right action $\id\times \mu:\mathcal{G}\times \mathcal{G}\times \mathcal{H}\to
\mathcal{G}\times \mathcal{G}$ of $\mathcal{H}$ on $\mathcal{G}\times \mathcal{G}$. If $M\in
{\rm sCoh}_{\rm f}^{(\mathcal{H},\Psi)}(\mathcal{G})$, it is clear that $X\boxtimes M\in \text{sCoh}_{\rm f}(\mathcal{G}\times \mathcal{G})$ is an $(\mathcal{H},\Psi)$-equivariant coherent sheaf on $\mathcal{G}\times \mathcal{G}$ (here we identify $\mathcal{H}$ with the supergroup subscheme  $\{1\}\times \mathcal{H}\subset \mathcal{G}\times \mathcal{G}$). But since ${\rm m}:\mathcal{G}\times \mathcal{G}\to \mathcal{G}$ is $\mathcal{H}$-equivariant, ${\rm m}_*$
carries $(\mathcal{H},\Psi)$-equivariant coherent sheaves  on $\mathcal{G}\times \mathcal{G}$ to $(\mathcal{H},\Psi)$-equivariant coherent sheaves on $\mathcal{G}$.
\end{proof}

\section{Exact module categories over ${\rm sCoh}_{\rm f}(\mathcal{G})$}\label{exmodcatscoh}

In this section we extend \cite[Section 3.3]{G} to the super case.

Let $\mathcal{G}$, $\mathcal{H}$, $\iota$ and $\Psi$ be as in \ref{equiqcohsh}. Set 
$$
\mathcal{M}=\mathcal{M}(\mathcal{H},\Psi):={\rm sCoh}_{\rm f}^{(\mathcal{H},\Psi)}(\mathcal{G}),\,\,\,\mathcal{M}^{\circ}=\mathcal{M}^{\circ}(\mathcal{H},\Psi):={\rm Coh}_{\rm f}^{(\mathcal{H},\Psi)}(\mathcal{G}),
$$
and let
$$\mathcal{V}=\mathcal{V}(\mathcal{H},\Psi):={\rm Comod}_{\Pro({\rm sCoh}_{\rm f}(\mathcal{G}))}(\widehat{\mathcal{O}(\mathcal{H})_{\Psi}}),$$
$$\mathcal{V}^{\circ}=\mathcal{V}^{\circ}(\mathcal{H},\Psi):={\rm Comod}_{\Pro({\rm Coh}_{\rm f}(\mathcal{G}))}(\widehat{\mathcal{O}(\mathcal{H})_{\Psi}})$$ be the abelian
categories of right comodules over $\widehat{\mathcal{O}(\mathcal{H})_{\Psi}}$ in  
$\Pro({\rm sCoh}_{\rm f}(\mathcal{G}))$ and $\Pro({\rm Coh}_{\rm f}(\mathcal{G}))$, respectively.

Let  
$$\delta=\delta_{\mathcal{H}}:=\iota_*\mathcal{O}(\mathcal{H})\in {\rm sCoh}_{\rm f}^{(\mathcal{H},\Psi)}(\mathcal{G}),\,\,\,\delta^{\circ}=\delta^{\circ}_{\mathcal{H}}:=\iota_*\mathcal{O}(\mathcal{H})\in {\rm Coh}_{\rm f}^{(\mathcal{H},\Psi)}(\mathcal{G}).$$ 

\begin{proposition}\label{grsch}
The following hold:
\begin{enumerate}
\item
The bifunctors
$$\ot^{\mathcal{M}}:{\rm sCoh}_{\rm f}(\mathcal{G})\boxtimes \mathcal{M}\to
\mathcal{M},\,\, X\boxtimes M\mapsto {\rm m}_*(X\boxtimes M)$$
and
$$\ot^{\mathcal{M}^{\circ}}:{\rm sCoh}_{\rm f}(\mathcal{G})\boxtimes \mathcal{M}^{\circ}\to
\mathcal{M}^{\circ},\,\, X\boxtimes M\mapsto {\rm m}_*(X\boxtimes M)$$
define on $\mathcal{M}$ and $\mathcal{M}^{\circ}$  structures of indecomposable   
${\rm sCoh}_{\rm f}(\mathcal{G})$-module categories.
\item 
The bifunctors
$$\ot^{\mathcal{V}}:{\rm sCoh}_{\rm f}(\mathcal{G})\boxtimes \mathcal{V}
\to \mathcal{V},\,\, X\boxtimes V\mapsto {\rm m}_*(X\boxtimes V)$$
and
$$\ot^{\mathcal{V}^{\circ}}:{\rm sCoh}_{\rm f}(\mathcal{G})\boxtimes \mathcal{V}^{\circ}
\to \mathcal{V}^{\circ},\,\, X\boxtimes V\mapsto {\rm m}_*(X\boxtimes V)$$
define on $\mathcal{V}$ and $\mathcal{V}^{\circ}$ structures of ${\rm sCoh}_{\rm f}(\mathcal{G})$-module categories.
\item
We have equivalences $\mathcal{M}\cong \mathcal{V}$ and $\mathcal{M}^{\circ}\cong \mathcal{V}^{\circ}$ of 
module categories over ${\rm sCoh}_{\rm f}(\mathcal{G})$. In particular, $\overline{\Hom}(\delta,\delta)\cong\widehat{\mathcal{O}(\mathcal{H})_{\Psi}}$ as supercoalgebras in $\Pro({\rm sCoh}_{\rm f}(\mathcal{G}))$, and $\overline{\Hom}(\delta^{\circ},\delta^{\circ})\cong\widehat{\mathcal{O}(\mathcal{H})_{\Psi}}$ as coalgebras in $\Pro({\rm Coh}_{\rm f}(\mathcal{G}))$.
\end{enumerate}
\end{proposition}

\begin{proof}
We prove it for $\mathcal{M}$ and $\mathcal{V}$, the proof for $\mathcal{M}^{\circ}$ and $\mathcal{V}^{\circ}$ being similar.

(1) Since ${\rm m}({\rm m}\times \id)={\rm m}(\id\times {\rm m})$ and $\Psi$ is an even $2$-cocycle,
it follows from Lemma \ref{simpleex} that $\ot^{\mathcal{M}}$
defines on $\mathcal{M}$ a structure of a ${\rm sCoh}_{\rm f}(\mathcal{G})$-module
category. Clearly, ${\rm sCoh}_{\rm f}(\mathcal{H})\subset{\rm sCoh}_{\rm f}(\mathcal{G})$ consists of those objects $X$ for which $X\ot ^{\mathcal{M}} \delta $ is a sum of multiples of $\delta$ and $k^{0|1}\ot \delta$, and any object $M\in \mathcal{M}$ is of the form $X\ot ^{\mathcal{M}}\delta$ for some $X\in {\rm sCoh}_{\rm f}(\mathcal{G})$. In particular, the simple object $\delta$ (see Proposition \ref{simpleex}) generates $\mathcal{M}$, so $\mathcal{M}$ is indecomposable.

(2) By definition, an object in $\mathcal{V}$ is a pair $(V,\rho_V)$ consisting of an object
$V\in \Pro({\rm sCoh}_{\rm f}(\mathcal{G}))$
and a morphism $\rho_V:V\to V\ot\widehat{\mathcal{O}(\mathcal{H})_{\Psi}}$ in
$\Pro({\rm sCoh}_{\rm f}(\mathcal{G}))$ satisfying the comodule axioms. It is clear that for every 
$X\in{\rm sCoh}_{\rm f}(\mathcal{G})$, we have ${\rm m}_*(X\boxtimes V)\in\Pro({\rm sCoh}_{\rm f}(\mathcal{G}))$  and that $\rho_{{\rm m}_*(X\boxtimes
V)}:=\id_X\ot\rho_V$ is a morphism in $\Pro({\rm sCoh}_{\rm f}(\mathcal{G}))$ defining on
${\rm m}_*(X\boxtimes V)$ a structure of a right comodule over
$\widehat{\mathcal{O}(\mathcal{H})_{\Psi}}$.

(3) Follows from Lemma \ref{eqcscom}.
\end{proof}

\begin{example}\label{impexs}
Let $\mathcal{G}$ be an affine supergroup scheme over $k$. 
\begin{enumerate}
\item
$\mathcal{M}(\{1\},1)={\rm sCoh}_{\rm f}(\mathcal{G})$ (the regular module).
\item
$\mathcal{M}^{\circ}(\{1\},1)={\rm Coh}_{\rm f}(\mathcal{G})$. 
\item
$\mathcal{M}(\mathcal{G},1)=\text{sVect}$ (the standard superfiber functor on ${\rm sCoh}_{\rm f}(\mathcal{G})$).
\item
$\mathcal{M}^{\circ}(\mathcal{G},1)=\text{Vect}$ (the standard fiber functor on ${\rm sCoh}_{\rm f}(\mathcal{G})$).
\end{enumerate}
\end{example}

\begin{proposition}\label{indecmod}
The indecomposable module categories $\mathcal{M}(\mathcal{H},\Psi)$ and $\mathcal{M}^{\circ}(\mathcal{H},\Psi)$ over ${\rm sCoh}_{\rm f}(\mathcal{G})$ are exact.
\end{proposition}

\begin{proof}
We prove it for $\mathcal{M}(\mathcal{H},\Psi)$, the proof for $\mathcal{M}^{\circ}(\mathcal{H},\Psi)$ being similar.

Set $\mathcal{M}:=\mathcal{M}(\mathcal{H},\Psi)$. It suffices to show that for every projective $P\in\Pro({\rm sCoh}_{\rm f}(\mathcal{G}))$ and $X\in \mathcal{M}$, $P\ot ^{\mathcal{M}} X$ is projective (see \ref{Module categories over tensor categories}). Clearly, it suffices to show it for $X:=\delta=\delta_{(\mathcal{H},\Psi)}$. Moreover, since any projective in $\Pro({\rm sCoh}_{\rm f}(\mathcal{G}))$ is a completed direct sum of $P_{g,\pm}$ (see \ref{The tensor category sCohfG}), it suffices to check that $P_g\ot ^{\mathcal{M}} \delta$ is projective. Furthermore, since $P_g=\delta_g\otimes P_1$, and $\delta_g\ot ^{\mathcal{M}} \,?$ is an autoequivalence of $\mathcal{M}$ as an abelian category (since $\delta_g$ is invertible), it suffices to do so for $g=1$. Finally, this is done just by computing this product explicitly using the definition, which yields that $P_1\ot ^{\mathcal{M}} \delta=\widehat{\mathcal{O}(\mathcal{H})_1}\otimes_k P(\delta)$, where $P(\delta)$ is the projective cover of $\delta$ (i.e., the unique indecomposable projective in the block of $\Pro(\mathcal{M})$ containing $\delta$; as a sheaf  on $\mathcal{G}$, it is the function algebra on the formal neighborhood of $\mathcal{H}$), and hence projective as desired.
\end{proof}

We say that two pairs $(\mathcal{H},\Psi)$ and $(\mathcal{H}',\Psi')$ are {\em conjugate} if there exists $g\in \mathcal{G}_0(k)$ such that $g\mathcal{H}g^{-1}=\mathcal{H}'$ and $\Psi^g=\Psi'$ in $H^2(\mathcal{H}',\mathbb{G}_m)$.

\begin{lemma}\label{indecmodeqclass}
If $(\mathcal{H},\Psi)$ and $(\mathcal{H}',\Psi')$ are conjugate then 
$$\mathcal{M}(\mathcal{H},\Psi)\cong\mathcal{M}(\mathcal{H}',\Psi')\,\,\,\text{and}\,\,\,\mathcal{M}^{\circ}(\mathcal{H},\Psi)\cong\mathcal{M}^{\circ}(\mathcal{H}',\Psi')$$ as module categories over ${\rm sCoh}_{\rm f}(\mathcal{G})$.
\end{lemma}

\begin{proof}
Follows from Proposition \ref{grsch} since for every $g\in \mathcal{G}_0(k)$, we have $\widehat{\mathcal{O}(\mathcal{H})_{\Psi}}\cong \widehat{\mathcal{O}(g\mathcal{H}g^{-1})_{\Psi^g}}$ as supercoalgebras in $\Pro({\rm sCoh}_{\rm f}(\mathcal{G}))$. 
\end{proof}

We are now ready to state and prove the main result of
this section.

\begin{theorem}\label{grsch1}
Let $\mathcal{G}$ be an affine supergroup scheme over $k$. There is
a $1:2$ correspondence between conjugacy classes of pairs $(\mathcal{H},\Psi)$ and
equivalence classes of indecomposable exact module categories over
${\rm sCoh}_{\rm f}(\mathcal{G})$, assigning $(\mathcal{H},\Psi)$ to $\mathcal{M}(\mathcal{H},\Psi)$ and $\mathcal{M}^{\circ}(\mathcal{H},\Psi)$.
\end{theorem}

\begin{proof}
By Proposition \ref{indecmod} and Lemma \ref{indecmodeqclass}, it remains to show that any indecomposable exact module category $\mathcal{M}$ over ${\rm sCoh}_{\rm f}(\mathcal{G})$ is of the form $\mathcal{M}(\mathcal{H},\Psi)$ or $\mathcal{M}^{\circ}(\mathcal{H},\Psi)$ for a unique pair $(\mathcal{H},\Psi)$ (up to conjugation). To this end, pick a simple object $\delta$ in $\mathcal{M}$, and let $\delta^{-}:=k^{0|1}\ot ^{\mathcal{M}} \delta$. There are two cases: either $\delta\ncong \delta^-$ or $\delta\cong \delta^-$.

Assume that $\delta\ncong \delta^-$. Consider the full subcategory
$$\mathcal{C}:=\{X\in {\rm sCoh}_{\rm f}(\mathcal{G})\mid X\ot ^{\mathcal{M}}\delta=\dim_k(X_0)\delta \oplus \dim_k(X_1) \delta^-\}\subset {\rm sCoh}_{\rm f}(\mathcal{G}).
$$
For every $X,Y\in \mathcal{C}$, we have 
\begin{eqnarray*}
\lefteqn{(X\ot Y)\ot ^{\mathcal{M}}\delta\cong
X\ot ^{\mathcal{M}}(Y\ot ^{\mathcal{M}}\delta)}\\
& \cong & X\ot ^{\mathcal{M}}(\dim_k(Y_0)\delta \oplus \dim_k(Y_1) \delta^-)\\
& \cong & \dim_k(Y_0)X\ot ^{\mathcal{M}}\delta \oplus 
\dim_k(Y_1)\Pi(X)\ot ^{\mathcal{M}}\delta\\
& \cong & \dim_k(Y_0)(\dim_k(X_0)\delta \oplus \dim_k(X_1) \delta^-)\oplus \\
& & 
\dim_k(Y_1)(\dim_k(X_1)\delta \oplus \dim_k(X_0) \delta^-)\\
& = & \dim_k(Y_0\ot X_0\oplus Y_1\ot X_1)\delta\oplus
\dim_k(Y_0\ot X_1\oplus Y_1\ot X_0)\delta^-. 
\end{eqnarray*}
Thus, $\mathcal{C}\subset {\rm sCoh}_{\rm f}(\mathcal{G})$ is a tensor subcategory. Moreover, the functor
$$F:\mathcal{C}\to \Vect,\,\,\,F(X)=\Hom_{\mathcal{M}}(\delta\oplus \delta^-,X\ot ^{\mathcal{M}} \delta),$$ together with the tensor structure
$F(\cdot)\otimes F(\cdot)\xrightarrow{\cong} F(\cdot\otimes \cdot)$
coming from the associativity constraint, is a fiber functor on $\mathcal{C}$. 

Now by Lemma \ref{tensubcat}, either  
$\mathcal{C}={\rm sCoh}_{\rm f}(\mathcal{H})$ for some closed {\em supergroup} subscheme $\mathcal{H}\subset\mathcal{G}$, or 
$\mathcal{C}={\rm Coh}_{\rm f}(\mathcal{H})$ for some closed {\em group} subscheme $\mathcal{H}\subset\mathcal{G}_0$. In any case, we see that for every $X\in\C$, $F(X)=\overline{X}$ is the underlying vector space of $X$. Thus we have a functorial isomorphism
$\overline{X}\otimes \overline{Y}\xrightarrow{\cong} \overline{X\otimes Y}$, which is nothing but an  
invertible even element $\Psi$ of $\mathcal{O}(\mathcal{H})^{\otimes 2}$, taking values in
$\mathbb{G}_m(k)$. Clearly, $\Psi$ is a twist for $\mathcal{O}(\mathcal{H})$.

Thus, we have obtained that if $\C={\rm sCoh}_{\rm f}(\mathcal{H})$ then the $\C$-submodule category $\langle \delta \rangle\subset \mathcal{M}$  
is equivalent to ${\rm sCoh}_{\rm f}^{(\mathcal{H},\Psi)}(\mathcal{H})$, and if $\C={\rm Coh}_{\rm f}(\mathcal{H})$ then the $\C$-submodule category $\langle \delta\rangle\subset \mathcal{M}$ is equivalent to ${\rm Coh}_{\rm f}^{(\mathcal{H},\Psi)}(\mathcal{H})$.
 
For $X\in {\rm sCoh}_{\rm f}(\mathcal{G})$, 
let $X_{\mathcal{H}}\in \C$ be the maximal subsheaf of $X$ which is scheme-theoretically supported on $\mathcal{H}$ (i.e., $\overline{X_{\mathcal{H}}}$ consists of all vectors in $\overline{X}$ which are annihilated by the defining ideal of $\mathcal{H}$ in $\mathcal{O}(\mathcal{G})$). Now, on the one hand, since for any $g\in \mathcal{G}_0(k)$, $\delta_g\ot ^{\mathcal{M}}\delta$ and $\delta_g\ot ^{\mathcal{M}}\delta^-$ are simple, and one of them is isomorphic to $\delta$ and the other one to $\delta^-$ if and only if $g\in \mathcal{H}_0(k)$, it is clear that
\begin{eqnarray*}
\lefteqn{\Hom_{\Pro({\rm sCoh}_{\rm f}(\mathcal{G}))}
(\overline{\Hom}(\delta\oplus \delta^-,\delta\oplus \delta^-),X)}\\
& = & \Hom_{\mathcal{M}}(\delta\oplus \delta^-,X\ot
^{\mathcal{M}}(\delta\oplus \delta^-))=\overline{X_{\mathcal{H}}}
\end{eqnarray*} 
(since it holds for any simple $X$). On the other hand, it is clear that $$\Hom_{\Pro({\rm sCoh}_{\rm f}(\mathcal{G}))}
(\widehat{\mathcal{O}(\mathcal{H})_{\Psi}},X)=\overline{X_{\mathcal{H}}}.$$
Thus by Yoneda's lemma, the two supercoalgebras $\overline{\Hom}(\delta\oplus \delta^-,\delta\oplus \delta^-)$ and $\widehat{\mathcal{O}(\mathcal{H})_{\Psi}}$ are isomorphic in $\Pro({\rm sCoh}_{\rm f}(\mathcal{G}))$. This implies that $\mathcal{M}$ is equivalent
to ${\rm Comod}_{\Pro({\rm sCoh}_{\rm f}(\mathcal{G}))}(\widehat{\mathcal{O}(\mathcal{H})_{\Psi}})$ as a module category over ${\rm sCoh}_{\rm f}(\mathcal{G})$ (as $\mathcal{M}$ is
indecomposable, exact, and generated by $\delta\oplus \delta^-$), hence to $\mathcal{M}(\mathcal{H},\Psi)$ by Proposition \ref{grsch}, as desired. 

Furthermore, the conjugacy class of $(\mathcal{H},\Psi)$ is uniquely determined by $\mathcal{M}$ since replacing $\delta$ with $\delta_g\ot ^{\mathcal{M}} \delta$, $g\in \mathcal{G}_0(k)$, corresponds to replacing $\mathcal{H}$ with $g\mathcal{H}g^{-1}$ and $\Psi$ with $\Psi^g$.

Finally, if $\delta\cong \delta^-$ then the proof that $\mathcal{M}\cong \mathcal{M}^{\circ}(\mathcal{H},\Psi)$ as module categories over ${\rm sCoh}_{\rm f}(\mathcal{G})$ for a uniquely determined conjugacy class of pairs, is similar.
\end{proof}

\begin{example}
Since for an affine group scheme $\mathcal{G}$ over $k$, we have ${\rm sCoh}_{\rm f}(\mathcal{G})={\rm Coh}_{\rm f}(\mathcal{G})\boxtimes {\rm sVect}$, we see that Theorem \ref{grsch1} reduces to \cite[Theorem 3.9]{G} in the even case.
\end{example}

\begin{example}\label{ex2}
Let $V$ be a $n$-dimensional odd $k$-vector space, $n\ge 0$. 
By Theorem \ref{grsch1}, equivalence classes of indecomposable exact module categories over $\text{sCoh}_{\rm f}(V)$ are in $2:1$ correspondence with equivalence classes of pairs $(W,B)$, where $W\subset V$ is a super subspace and $B\in S^2W^*$. For example, if $n=0$ then there are two non-equivalent indecomposable exact module categories over $\text{sCoh}_{\rm f}(V)=\text{sVect}$: $\text{Vect}$ and $\text{sVect}$.
Also, if $n=1$ then there are exactly three non-equivalent pairs of the form $(W,B)$: $(0,0)$, $(V,0)$ and $(V,B)$, where $B(v,v)=1$ ($v$ a fixed basis for $V$). Thus, there are six non-equivalent indecomposable exact module categories over $\text{sCoh}_{\rm f}(V)$ (in agreement with \cite[Theorem 4.5]{EO}). More precisely, we have $\mathcal{M}(0,0)\cong{\rm sCoh}_{\rm f}(V)$ and $\mathcal{M}^{\circ}(0,0)\cong{\rm Coh}_{\rm f}(V)$, $\mathcal{M}(V,0)$,   $\mathcal{M}(V,B)$, which are semisimple of rank $2$, and $\mathcal{M}^{\circ}(V,0)$, $\mathcal{M}^{\circ}(V,B)$, which are semisimple of rank $1$.
\end{example}

\begin{remark}\label{scohfgomega2}
Retain the notation from Remark \ref{scohfgomega1}.
Similarly, the categories ${\rm sCoh}_{\rm f}^{(\mathcal{H},\Psi)}(\mathcal{G},\Omega)$ and ${\rm Coh}_{\rm f}^{(\mathcal{H},\Psi)}(\mathcal{G},\Omega)$ admit a structure of an indecomposable exact module category over ${\rm sCoh}_{\rm f}(\mathcal{G},\Omega)$ given by convolution of sheaves, and furthermore, there is a $1:2$ correspondence between (appropriately defined) conjugacy classes of pairs $(\mathcal{H},\Psi)$ and
equivalence classes of indecomposable exact module categories over
${\rm sCoh}_{\rm f}(\mathcal{G},\Omega)$, assigning $(\mathcal{H},\Psi)$ to ${\rm sCoh}_{\rm f}^{(\mathcal{H},\Psi)}(\mathcal{G},\Omega)$ and ${\rm Coh}_{\rm f}^{(\mathcal{H},\Psi)}(\mathcal{G},\Omega)$.
\end{remark}

\section{Exact module categories over $\sRep(\mathcal{G})$}\label{repg}
In this section we extend \cite[Section 4]{G} to the super case.

Let $\mathcal{C}$ be a tensor category. Given two exact module categories $\mathcal{M}$, $\mathcal{N}$ over $\mathcal{C}$, let $\Fun_{\C}(\mathcal{M},\mathcal{N})$ denote the abelian category of $\mathcal{C}$-functors from $\mathcal{M}$ to $\mathcal{N}$. The \emph{dual category of $\mathcal{C}$ with respect to $\mathcal{M}$} is the category $\mathcal{C}^*_{\mathcal{M}}:=\End_{\mathcal{C}}(\mathcal{M})$ of $\mathcal{C}$-endofunctors of $\mathcal{M}$. If $\mathcal{M}$ is indecomposable, $\mathcal{C}^*_{\mathcal{M}}$ is a tensor category, and $\mathcal{M}$ is an indecomposable exact module category over $\mathcal{C}^*_{\mathcal{M}}$. Also, $\Fun_{\C}(\mathcal{M},\mathcal{N})$ is an exact module category over
$\mathcal{C}^*_{\mathcal{M}}$ via the composition of functors.

\subsection{Module categories.}\label{Module categories} 

Retain the notation from \ref{equiqcohsh} and \ref{exmodcatscoh}. Set 
$$
\mathcal{M}((\mathcal{G},1),(\mathcal{H},\Psi)):={\rm sCoh}_{\rm f}^{((\mathcal{G},1),(\mathcal{H},\Psi))}(\mathcal{G})$$
and
$$ 
\mathcal{M}^{\circ}((\mathcal{G},1),(\mathcal{H},\Psi)):={\rm Coh}_{\rm f}^{((\mathcal{G},1),(\mathcal{H},\Psi))}(\mathcal{G}).
$$ 

Recall that the $2$-cocycle $\Psi$ determines a central extension
$\mathcal{H}_{\Psi}$ of $\mathcal{H}$ by $\mathbb{G}_m$. By an {\em $(\mathcal{H},\Psi)$-superrepresentation} of $\mathcal{H}$ we will mean a rational representation of the affine supergroup scheme
$\mathcal{H}_{\Psi}$ on a $k$-supervector space on which $\mathbb{G}_m$ acts with weight $1$
(i.e., via the identity character). Let us denote the category of finite dimensional $(\mathcal{H},\Psi)$-superrepresentations of $\mathcal{H}_{\Psi}$ by $\mathcal{N}(\mathcal{H},\Psi)$. Clearly, we have an equivalence of abelian categories
$$\mathcal{N}(\mathcal{H},\Psi)\cong{\rm sComod}(\mathcal{O}(\mathcal{H})_{\Psi}).$$
Similarly, let $\mathcal{N}^{\circ}(\mathcal{H},\Psi)$ be the category of finite dimensional $(\mathcal{H},\Psi)$-{\em representations} of $\mathcal{H}_{\Psi}$. We have an equivalence of abelian categories
$$\mathcal{N}^{\circ}(\mathcal{H},\Psi)\cong{\rm Comod}(\mathcal{O}(\mathcal{H})_{\Psi}).$$

\begin{lemma}\label{modrep}
The following hold:
\begin{enumerate}
\item
We have abelian equivalences 
$$\Fun_{{\rm sCoh}_{\rm f}(\mathcal{G})}(\mathcal{M}^{\circ}(\mathcal{G},1),\mathcal{M}(\mathcal{H},\Psi))\cong {\mathcal{M}^{\circ}((\mathcal{G},1),(\mathcal{H},\Psi))}$$ 
and
$$\Fun_{{\rm sCoh}_{\rm f}(\mathcal{G})}(\mathcal{M}^{\circ}(\mathcal{G},1),\mathcal{M}^{\circ}(\mathcal{H},\Psi))\cong {\mathcal{M}((\mathcal{G},1),(\mathcal{H},\Psi))}.$$
In particular, we have a tensor equivalence $${\rm sCoh}_{\rm f}(\mathcal{G})^*_{\mathcal{M}^{\circ}(\mathcal{G},1)}\cong \sRep(\mathcal{G}).$$

\item
We have $\sRep(\mathcal{G})$-module equivalences
$$\Fun_{{\rm sCoh}_{\rm f}(\mathcal{G})}(\mathcal{M}^{\circ}(\mathcal{G},1),\mathcal{M}(\mathcal{H},\Psi))\cong \mathcal{N}^{\circ}(\mathcal{H},\Psi)$$
and
$$\Fun_{{\rm sCoh}_{\rm f}(\mathcal{G})}(\mathcal{M}^{\circ}(\mathcal{G},1),\mathcal{M}^{\circ}(\mathcal{H},\Psi))\cong \mathcal{N}(\mathcal{H},\Psi).$$
\end{enumerate}
\end{lemma}

\begin{proof}
We prove the theorem for functors to $\mathcal{M}(\mathcal{H},\Psi)$, the proof for functors to $\mathcal{M}^{\circ}(\mathcal{H},\Psi)$ being similar.

(1) Since $\mathcal{M}^{\circ}(\mathcal{G},1)=\Vect$, a functor $\mathcal{M}^{\circ}(\mathcal{G},1)\to\mathcal{M}(\mathcal{H},\Psi)$ is
just an $(\mathcal{H},\Psi)$-equivariant sheaf $X$ on $\mathcal{G}$. The fact that the
functor is a ${\rm sCoh}_{\rm f}(\mathcal{G})$-module functor means that we have functorial isomorphisms $\mu_S:\overline{S}\ot X\xrightarrow{\cong}S\ot X$ in $\mathcal{M}(\mathcal{H},\Psi)$, $S\in {\rm sCoh}_{\rm f}(\mathcal{G})$. Thus, $\mu$ gives $X$ a commuting
$\mathcal{G}$-equivariant structure for the left action of $\mathcal{G}$ on itself, i.e., $X$ is a $((\mathcal{G},1),(\mathcal{H},\Psi))$-biequivariant sheaf  on $\mathcal{G}$. In particular, for $S=k^{0\mid 1}$, we have an isomorphism $\mu_{k^{0\mid 1}}:X_0\xrightarrow{\cong}X_1$, hence $X$ corresponds to $((\mathcal{G},1),(\mathcal{H},\Psi))$-biequivariant $\mathcal{O}(\mathcal{G})$-module, as desired.

Conversely, it is clear that any $((\mathcal{G},1),(\mathcal{H},\Psi))$-biequivariant $\mathcal{O}(\mathcal{G})$-module $X_0$  
defines a ${\rm sCoh}_{\rm f}(\mathcal{G})$-module functor $\mathcal{M}^{\circ}(\mathcal{G},1)\to\mathcal{M}(\mathcal{H},\Psi)$ determined by $k\mapsto X$ with $X_0=X_1$.

Finally, the category of $(\mathcal{G},\mathcal{G})$-biequivariant sheaves
on $\mathcal{G}$ is equivalent to the category $\sRep(\mathcal{G})$ as a tensor category, and the second claim follows.

(2) If $X$ is a $((\mathcal{G},1),(\mathcal{H},\Psi))$-biequivariant $\mathcal{O}(\mathcal{G})$-module, then the
inverse image sheaf  ${\rm e}^{*}(X)$ on $\Spec(k)$ (``the stalk at $1$") acquires a structure of an $(\mathcal{H},\Psi)$-representation via the action
of the element $(h,h^{-1})$ in $\mathcal{G}\times \mathcal{H}$, i.e., it is an object in
$\mathcal{N}^{\circ}(\mathcal{H},\Psi)$. We have thus defined a functor $${\mathcal{M}^{\circ}((\mathcal{G},1),(\mathcal{H},\Psi))}\to \mathcal{N}^{\circ}(\mathcal{H},\Psi),\,\,X\mapsto {\rm e}^{*}(X).$$

Conversely, an $(\mathcal{H},\Psi)$-representation $V$ can be spread out over
$\mathcal{G}$ and made into a $(\mathcal{G},(\mathcal{H},\Psi))$-biequivariant $\mathcal{O}(\mathcal{G})$-module. In other words, we have the functor $$\mathcal{N}^{\circ}(\mathcal{H},\Psi)\to {\mathcal{M}^{\circ}((\mathcal{G},1),(\mathcal{H},\Psi))},\,\,V\mapsto \mathcal{O}(\mathcal{G})\otimes_k V.$$

Finally, it is straightforward to verify that the two functors constructed above are inverse to each other.
\end{proof}

Similarly, we have the following result.

\begin{lemma}\label{modrep1}
The following hold:
\begin{enumerate}
\item
We have abelian equivalences 
$$\Fun_{{\rm sCoh}_{\rm f}(\mathcal{G})}(\mathcal{M}(\mathcal{G},1),\mathcal{M}(\mathcal{H},\Psi))\cong {\mathcal{M}((\mathcal{G},1),(\mathcal{H},\Psi))}$$ 
and
$$\Fun_{{\rm sCoh}_{\rm f}(\mathcal{G})}(\mathcal{M}(\mathcal{G},1),\mathcal{M}^{\circ}(\mathcal{H},\Psi))\cong {\mathcal{M}^{\circ}((\mathcal{G},1),(\mathcal{H},\Psi))}.$$
In particular, we have a tensor equivalence $${\rm sCoh}_{\rm f}(\mathcal{G})^*_{\mathcal{M}(\mathcal{G},1)}\cong \sRep(\mathcal{G}).$$

\item
We have $\sRep(\mathcal{G})$-module equivalences
$$\Fun_{{\rm sCoh}_{\rm f}(\mathcal{G})}(\mathcal{M}(\mathcal{G},1),\mathcal{M}(\mathcal{H},\Psi))\cong \mathcal{N}(\mathcal{H},\Psi)$$
and
$$\Fun_{{\rm sCoh}_{\rm f}(\mathcal{G})}(\mathcal{M}(\mathcal{G},1),\mathcal{M}^{\circ}(\mathcal{H},\Psi))\cong \mathcal{N}^{\circ}(\mathcal{H},\Psi).$$ \qed
\end{enumerate}
\end{lemma}

\begin{example}
We have the following:
\begin{enumerate}
\item
$\mathcal{N}(\{1\},1)=\text{sVect}$ is the usual superfiber functor on $\sRep(\mathcal{G})$.
\item
$\mathcal{N}^{\circ}(\{1\},1)=\Vect$ is the usual fiber functor on $\sRep(\mathcal{G})$.
\end{enumerate} 
\end{example}

\begin{lemma}\label{modfun1}
The following hold:
\begin{enumerate}
\item
We have a tensor equivalence 
$$\sRep(\mathcal{G})_{\mathcal{N}^{\circ}(\{1\},1)}^*\cong{\rm sCoh}_{\rm f}(\mathcal{G}).$$ 
\item
We have ${\rm sCoh}_{\rm f}(\mathcal{G})$-module equivalences
$$\Fun_{\sRep(\mathcal{G})}\left(\mathcal{N}^{\circ}(\{1\},1),\mathcal{N}(\mathcal{H},\Psi)\right)\cong \mathcal{M}^{\circ}(\mathcal{H},\Psi)$$
and
$$\Fun_{\sRep(\mathcal{G})}\left(\mathcal{N}^{\circ}(\{1\},1),\mathcal{N}^{\circ}(\mathcal{H},\Psi)\right)\cong \mathcal{M}(\mathcal{H},\Psi).$$
\end{enumerate}
\end{lemma}

\begin{proof}
The proof is similar to the proof of Lemma \ref{modrep}.
\end{proof}

Lemma \ref{modfun1} prompts the following definition.

\begin{definition}
An indecomposable exact module category $\mathcal{N}$ over
$\sRep(\mathcal{G})$ is called {\em geometrical} if $\Fun_{\sRep(\mathcal{G})}(\mathcal{N}^{\circ}(\{1\},1),\mathcal{N})\ne 0$.
\end{definition}

It is clear that geometrical module categories over
$\sRep(\mathcal{G})$ form a full $2$-subcategory $\Mod_{\rm geom}(\sRep(\mathcal{G}))$ of the $2$-category $\Mod(\sRep(\mathcal{G}))$. 

We can now deduce from Lemmas \ref{modrep}, \ref{modfun1} the main result of this section, which says that geometrical module categories over $\sRep(\mathcal{G})$ are precisely those exact module categories which come from exact module categories over ${\rm sCoh}_{\rm f}(\mathcal{G})$. More precisely, we have the following generalization of \cite[Theorem 4.5]{G}.

Recall that $(\mathcal{H},\Psi)$ and $(\mathcal{H}',\Psi')$ are {\em conjugate} if there exists $g\in \mathcal{G}_0(k)$ such that $g\mathcal{H}g^{-1}=\mathcal{H}'$ and $\Psi^g=\Psi'$ in $H^2(\mathcal{H}',\mathbb{G}_m)$.

\begin{theorem}\label{modg}
Let $\mathcal{G}$ be an affine supergroup scheme over $k$. Then the $2$-functors
$$\Mod({\rm sCoh}_{\rm f}(\mathcal{G}))\to \Mod_{\rm geom}(\sRep(\mathcal{G})),\,
\mathcal{M}\mapsto\Fun_{{\rm sCoh}_{\rm f}(\mathcal{G})}(\mathcal{M}^{\circ}(\mathcal{G},1),\mathcal{M}),$$ 
and
$$\Mod_{\rm geom}(\sRep(\mathcal{G}))\to \Mod({\rm sCoh}_{\rm f}(\mathcal{G})),\,
\mathcal{N}\mapsto\Fun_{\sRep(\mathcal{G})}(\mathcal{N}^{\circ}(\{1\},1),\mathcal{N}),$$
are inverse to each other. In particular, there is
a $1:2$ correspondence between conjugacy classes of pairs $(\mathcal{H},\Psi)$ and
equivalence classes of indecomposable geometrical module categories over
$\sRep(\mathcal{G})$, assigning $(\mathcal{H},\Psi)$ to $\mathcal{N}(\mathcal{H},{\Psi})$ and $\mathcal{N}^{\circ}(\mathcal{H},{\Psi})$. \qed
\end{theorem}

\begin{remark}\label{notall} 
If $\mathcal{G}$ is \emph{not} finite, $\sRep(\mathcal{G})$ may very well have non-geometrical module categories (see \cite[Remark 4.6]{G}).
\end{remark}

\begin{remark}\label{scohfgomega3}
Retain the notation from Remark \ref{scohfgomega2}.
Similarly to the even case \cite{G}, we can define supergroup scheme-theoretical categories $\C(\mathcal{G},\mathcal{H},\Omega,\Psi)$ and $\C ^{\circ}(\mathcal{G},\mathcal{H},\Omega,\Psi)$ as the dual categories of ${\rm sCoh}_{\rm f}(\mathcal{G},\Omega)$ with respect to ${\rm sCoh}_{\rm f}^{(\mathcal{H},\Psi)}(\mathcal{G},\Omega)$ and ${\rm Coh}_{\rm f}^{(\mathcal{H},\Psi)}(\mathcal{G},\Omega)$), respectively. We then have that $\C(\mathcal{G},\mathcal{H},\Omega,\Psi)$ is
equivalent to the tensor category of $((\mathcal{H},\Psi),(\mathcal{H},\Psi))$-biequivariant coherent sheaves on $(\mathcal{G},\Omega)$, supported on finitely many left $\mathcal{H}_0$-cosets (equivalently, right $\mathcal{H}_0$-cosets), with tensor product given by convolution of sheaves.
For example, the {\em center} $\mathcal{Z}({\rm sCoh}_{\rm f}(\mathcal{G}))$ of ${\rm sCoh}_{\rm f}(\mathcal{G})$ is supergroup scheme-theoretical since
$$
\mathcal{Z}({\rm sCoh}_{\rm f}(\mathcal{G}))\cong
({\rm sCoh}_{\rm f}(\mathcal{G})\boxtimes {\rm sCoh}_{\rm f}(\mathcal{G})^{\rm op})^*_{{\rm sCoh}_{\rm f}(\mathcal{G})},
$$
so 
$$\mathcal{Z}({\rm sCoh}_{\rm f}(\mathcal{G}))\cong\C(\mathcal{G}\times \mathcal{G},\mathcal{G},1,1)$$ as tensor categories, where $\mathcal{G}$ is viewed as a closed supergroup subscheme  of $\mathcal{G}\times \mathcal{G}$ via the diagonal morphism $\Delta:\mathcal{G}\to \mathcal{G}\times \mathcal{G}$.

Moreover, we can define indecomposable {\em geometrical} module categories over $\C:=\C(\mathcal{G},\mathcal{H},\Omega,\Psi)$, and obtain that the $2$-functors
$$\Mod({\rm sCoh}_{\rm f}(\mathcal{G},\Omega))\to \Mod_{\rm geom}(\C),\,\,
\mathcal{M}\mapsto\Fun_{{\rm sCoh}_{\rm f}(\mathcal{G},\Omega)}(\mathcal{M}(\mathcal{H},\Psi),\mathcal{M}),$$ 
and
$$\Mod_{\rm geom}(\C)\to \Mod({\rm sCoh}_{\rm f}(\mathcal{G},\Omega)),\,\,
\mathcal{N}\mapsto\Fun_{\C}(\mathcal{M}(\mathcal{H},\Psi),\mathcal{N}),$$
are $2$-equivalences which are inverse to each other. 
In particular, the
equivalence classes of geometrical module categories over $\C$ are in $2:1$ correspondence with the conjugacy classes of pairs $(\mathcal{H}',\Psi')$ such that
$\mathcal{H}'\subset \mathcal{G}$ is a closed supergroup subscheme  and $\Psi'\in
C^2(\mathcal{H}',\mathbb{G}_m)$ satisfies $d\Psi'=\Omega_{|\mathcal{H}'}$. (The analogs for $\C ^{\circ}(\mathcal{G},\mathcal{H},\Omega,\Psi)$ are obvious.)
\end{remark}

\subsection{Semisimple module categories of rank $1$.}\label{Semisimple module categories of rank $1$} Recall that the set of equivalence classes of semisimple module
categories over $\sRep(\mathcal{G})$ of rank $1$ is in bijection with the set
of equivalence classes of tensor
structures on the forgetful functor $\sRep(\mathcal{G})\to \Vect$. Therefore, Theorem \ref{modg} implies that the
conjugacy class of any pair $(\mathcal{H},\Psi)$ for which the category
${\rm sComod}(\mathcal{O}(\mathcal{H})_{\Psi})$ or ${\rm Comod}(\mathcal{O}(\mathcal{H})_{\Psi})$ is semisimple of rank $1$ gives rise
to an equivalence class of a tensor structure on the forgetful functor
$\sRep(\mathcal{G})\to \Vect$. Clearly, for such pair $(\mathcal{H},\Psi)$, $\mathcal{H}$ must be a
\emph{finite} supergroup subscheme of $\mathcal{G}$ (as a simple coalgebra must be finite dimensional). This observation suggests the following definition.

\begin{definition}\label{nondeg}
Let $\mathcal{H}$ be a finite supergroup scheme over $k$. We call an even $2$-cocycle $\Psi:\mathcal{H}\times \mathcal{H}\to \mathbb{G}_m$ (equivalently, a twist
$\Psi$ for $\mathcal{O}(\mathcal{H})=(k\mathcal{H})^*$) \emph{non-degenerate} if the
category ${\rm sComod}(\mathcal{O}(\mathcal{H})_{\Psi})$ or ${\rm Comod}(\mathcal{O}(\mathcal{H})_{\Psi})$ is equivalent to $\Vect$.
\end{definition}

We thus have the following corollary.

\begin{corollary}\label{twists}
The conjugacy class of a pair $(\mathcal{H},\Psi)$, where $\mathcal{H}\subset \mathcal{G}$ is a finite
closed supergroup subscheme  and $\Psi:\mathcal{H}\times \mathcal{H}\to \mathbb{G}_m$
is a non-degenerate even $2$-cocycle, gives rise to an equivalence class
of an even Hopf $2$-cocycle for $\mathcal{O}(\mathcal{G})$. \qed
\end{corollary}
 
\begin{remark}
Finite supergroup schemes having a non-degenerate even $2$-cocycle may be called supergroup schemes of \emph{central type} in analogy with the even case \cite[Remark 4.9]{G}.
\end{remark}

\subsection{Exact module categories over finite supergroup schemes}\label{fin} 
Thanks to \cite[Theorem 3.31]{EO}, Theorem \ref{modg} can be strengthened in the
finite case to give a canonical bijection between exact module
categories over ${\rm sCoh}_{\rm f}(\mathcal{G})={\rm sCoh}(\mathcal{G})$ and $\sRep(\mathcal{G})$ (i.e., for finite supergroup schemes, every exact module category over $\sRep(\mathcal{G})$ is geometrical). Namely, we have the following result.

\begin{theorem}\label{modgfin}
Let $\mathcal{G}$ be a finite supergroup scheme over $k$. The $2$-functors
$$\Mod({\rm sCoh}(\mathcal{G}))\to \Mod(\sRep(\mathcal{G})),\,\,
\mathcal{M}\mapsto\Fun_{{\rm sCoh}(\mathcal{G})}(\mathcal{M}^{\circ}(\mathcal{G},1),\mathcal{M}),$$ 
and
$$\Mod(\sRep(\mathcal{G}))\to \Mod({\rm sCoh}(\mathcal{G})),\,\,
\mathcal{N}\mapsto\Fun_{\sRep(\mathcal{G})}(\mathcal{N}^{\circ}(\{1\},1),\mathcal{N}),$$
are inverse to each other. In particular, the
equivalence classes of indecomposable exact module categories over
$\sRep(\mathcal{G})=\sRep(k\mathcal{G})$ are $2:1$ parameterized by the conjugacy classes of
pairs $(\mathcal{H},\Psi)$, where $\mathcal{H}\subset \mathcal{G}$ is a closed supergroup subscheme  and
$\Psi:\mathcal{H}\times \mathcal{H}\to \mathbb{G}_m$ is a normalized even $2$-cocycle. \qed
\end{theorem}

\begin{example}\label{ex2'}
Let $V$ be a one-dimensional odd vector space, and consider the purely odd finite supergroup scheme $\mathcal{G}:=V$. By Example \ref{ex2} and Theorem \ref{modgfin}, the tensor category $\text{sRep}(V)$ has exactly six non-equivalent indecomposable exact left module categories corresponding to the pairs $(0,0)$, $(V,0)$ and $(V,B)$, where $B(v,v)=1$. Namely, the categories $\mathcal{N}(0,0)={\rm sVect}$, $\mathcal{N}^{\circ}(0,0)=\Vect$, $\mathcal{N}(V,0)={\rm sMod}(\wedge V)$, $\mathcal{N}^{\circ}(V,0)={\rm Mod}(\wedge V)$, $\mathcal{N}(V,B)={\rm sMod}(k\mathbb{Z}_2)=\Vect$ (here $k(\mathbb{Z}/2\mathbb{Z})$ is viewed as a superalgebra, where the generator of $\mathbb{Z}/2\mathbb{Z}$ is odd), and $\mathcal{N}^{\circ}(V,B)={\rm Mod}(\mathbb{Z}/2\mathbb{Z})={\rm sVect}$. 
\end{example}

\section{The classification of triangular Hopf algebras with the Chevalley property}\label{The classification of triangular Hopf algebras with the Chevalley property}

In Sections \ref{Twists for kG}, \ref{Minimal twists for kG} we assume that $\mathcal{G}$ is a {\em finite} supergroup scheme over $k$. 
(The even case is treated in \cite[Sections 6.1-6.3]{G}.)

\subsection{Twists for $k\mathcal{G}$.}\label{Twists for kG} By \cite[Theorem 5.7]{AEGN}, there is a bijection between non-degenerate
twists for $k\mathcal{G}$ and non-degenerate twists for $\mathcal{O}(\mathcal{G})$. Hence, as a consequence of Theorem \ref{modgfin}, we
deduce the following strengthening of Corollary \ref{twists}.

\begin{corollary}\label{twistsfinite}
Let $\mathcal{G}$ be a finite supergroup scheme over $k$. The following four sets are in canonical bijection with each other:
\begin{enumerate}
\item
Equivalence classes of tensor structures
on the forgetful functor $\sRep(\mathcal{G})\to \Vect$.
\item 
Gauge equivalence classes of twists for $k\mathcal{G}$.
\item
Conjugacy classes of pairs $(\mathcal{H},\Psi)$, where
$\mathcal{H}\subset \mathcal{G}$ is a closed supergroup subscheme  and $\Psi:\mathcal{H}\times \mathcal{H}\to
\mathbb{G}_m$ is a non-degenerate even $2$-cocycle.
\item
Conjugacy classes of pairs $(\mathcal{H},\mathcal{J})$, where $\mathcal{H}\subset \mathcal{G}$
is a closed supergroup subscheme  and $\mathcal{J}$ is a  non-degenerate twist for $k\mathcal{H}$. \qed
\end{enumerate}
\end{corollary}

\begin{remark}
If moreover, $[\g_1,\g_1]=0$ (e.g., in characteristic $0$, or if $k\mathcal{G}=k\mathcal{G}_0\ltimes \wedge \g_1$), then each one of the above four sets is in bijection with the set of conjugacy classes of quadruples $(\mathcal{H}_0,\psi,\h_1,B)$, where $\mathcal{H}_0\subset \mathcal{G}_0$ is a closed subgroup scheme, $\psi:\mathcal{H}_0\times \mathcal{H}_0\to k$ is a non-degenerate $2$-cocycle, $Y\subset \g_1$ is an $\mathcal{H}_0$-invariant subspace, and $B\in S^2\h_1^*$ is non-degenerate (see \ref{gsch}).
\end{remark}

\begin{remark}\label{finite}
Corollary \ref{twistsfinite} was proved for \'{e}tale group schemes in \cite{Mov, EG, AEGN}, for finite supergroups $\mathcal{G}$ such that $k\mathcal{G}=k\mathcal{G}_0\ltimes \wedge \g_1$ in \cite{EO}, and for finite group schemes in \cite{G}.
\end{remark}

\begin{example}\label{ex2''}
Retain the notation from Example \ref{ex2'}. Then the tensor category $\text{sRep}(V)$ has exactly two non-equivalent fiber functors to $\Vect$ corresponding to the left module categories $\mathcal{N}^{\circ}(0,0)$ and $\mathcal{N}(V,B)$.
\end{example}

\subsection{Minimal twists for $k\mathcal{G}$.}\label{Minimal twists for kG} Recall that a twist $\mathcal{J}$ for $k\mathcal{G}$ is called \emph{minimal} if the
triangular Hopf superalgebra $((k\mathcal{G})^{\mathcal{J}},\mathcal{J}_{21}^{-1}\mathcal{J})$ is minimal, i.e., if
the left (right) tensorands of $\mathcal{J}_{21}^{-1}\mathcal{J}$ span $k\mathcal{G}$ \cite{R}. 

By \cite[Proposition 6.7]{G}, a twist for a finite group scheme is minimal if and only if it is non-degenerate. In this section we extend this result to the super case, using the following result (see \cite[Lemma A.8]{EG3} and \cite[Proposition 1]{B}).

\begin{proposition}\label{quot}
Let $\mathcal{D}$ and $\mathcal{E}$ be symmetric tensor categories over $k$, and suppose there exists a surjective\footnote{I.e., any object $X\in \mathcal{E}$ is isomorphic to a subquotient of $F(V)$ for some $V\in \mathcal{D}$.} symmetric tensor functor $F:\mathcal{D}\to \mathcal{E}$. If $\mathcal{D}$ is finitely tensor-generated and (super-)Tannakian, then so is $\mathcal{E}$. \qed
\end{proposition}

We can now state and prove the first main result of this section.

\begin{proposition}\label{minond}
Let $\mathcal{G}$ be a finite supergroup scheme over $k$, and let $\mathcal{J}$ be a twist for $k\mathcal{G}$. Then $\mathcal{J}$ is minimal if and only if it is non-degenerate.
\end{proposition}

\begin{proof}
Suppose $\mathcal{J}$ is minimal. By Corollary \ref{twistsfinite}, there exist a closed supergroup subscheme  $\mathcal{H}\subset\mathcal{G}$ and a non-degenerate twist
$\overline{\mathcal{J}}$ for $k\mathcal{H}$, such that the image of
$\overline{\mathcal{J}}$ under the embedding
$(k\mathcal{H})_{\overline{\mathcal{J}}}\hookrightarrow (k\mathcal{G})_{\mathcal{J}}$ is $\mathcal{J}$. Since
$\mathcal{J}$ is minimal and $\mathcal{H}\subset \mathcal{G}$, it follows that
$\mathcal{H}=\mathcal{G}$.

Conversely, suppose $\mathcal{J}$ is non-degenerate.
Let $(\mathcal{A},\mathcal{J}_{21}^{-1}\mathcal{J})$ be the minimal triangular Hopf sub-superalgebra of
$((k\mathcal{G})^{\mathcal{J}},\mathcal{J}_{21}^{-1}\mathcal{J})$. The restriction functor $\sRep(\mathcal{G})\surj \sRep(\mathcal{A})$ is a surjective symmetric tensor functor. Thus by Proposition \ref{quot},
$\sRep(\mathcal{A})$ is equivalent to $\sRep(\mathcal{H},u)$, as a symmetric tensor category, for some closed supergroup subscheme  $\mathcal{H}\subset\mathcal{G}$.
Now, it is a standard fact
(see, e.g., \cite{G1}) that such an equivalence functor gives rise to a twist $\mathcal{I}\in
(k\mathcal{H})^{\ot 2}$ and an isomorphism of triangular Hopf superalgebras
$
((k\mathcal{H})^\mathcal{I},\mathcal{I}_{21}^{-1}\mathcal{I})\xrightarrow{\cong} (\mathcal{A},\mathcal{J}_{21}^{-1}\mathcal{J})$.

We therefore get an injective homomorphism of triangular Hopf superalgebras
$((k\mathcal{H})^\mathcal{I},\mathcal{I}_{21}^{-1}\mathcal{I})\inj ((k\mathcal{G})^\mathcal{J},\mathcal{J}_{21}^{-1}\mathcal{J})$, which implies
that $\mathcal{J}\mathcal{I}^{-1}$ is a symmetric twist for $k\mathcal{G}$. But by \cite[Theorem 3.2]{DM}, this implies that $\mathcal{J}\mathcal{I}^{-1}$ is gauge equivalent to $1\ot 1$. Therefore, the triangular Hopf superalgebras
$((k\mathcal{G})^{\mathcal{J}\mathcal{I}^{-1}},\mathcal{I}_{21}\mathcal{J}_{21}^{-1}\mathcal{J}\mathcal{I}^{-1})$ and $(k\mathcal{G},1\ot 1)$ are
isomorphic. In other words, $((k\mathcal{G})^{\mathcal{I}},\mathcal{I}_{21}^{-1}\mathcal{I})$ and
$((k\mathcal{G})^{\mathcal{J}},\mathcal{J}_{21}^{-1}\mathcal{J})$ are isomorphic as triangular Hopf superalgebras, i.e., the pairs $(\mathcal{G},\mathcal{J})$ and $(\mathcal{H},\mathcal{I})$ are conjugate. We thus conclude from Corollary \ref{twistsfinite} that $\mathcal{H}=\mathcal{G}$, and hence that $\mathcal{J}$ is a minimal twist, as required.
\end{proof}

\begin{remark}
Corollary \ref{twistsfinite} and Proposition \ref{minond} extend \cite[Corollary 6.3 \& Proposition 6.7]{G} to the super case.
\end{remark}

\subsection{Triangular Hopf algebras}\label{Triangular Hopf algebras} 
Let $(H,R)$ be a finite dimensional triangular Hopf algebra with the Chevalley property over $k$. Recall that by \cite[Corollary 4.1]{EG3}, $(H,R)$ is {\em twist} equivalent to a finite dimensional triangular Hopf algebra with $R$-matrix of rank $\le 2$ (i.e., to a modified supergroup algebra \cite[Definition 3.3.4]{AEG}). Hence by \cite[Corollary 3.3.3]{AEG}, $(H,R)$ corresponds to a unique pair $(\mathcal{G},\epsilon)$, where $\mathcal{G}$ is a finite supergroup scheme over $k$ (see \ref{gsch}).  
Thus Corollary \ref{twistsfinite} implies the following classification result, which extends \cite[Theorem 5.1]{EG3} to arbitrary finite dimensional triangular Hopf algebras with the Chevalley property over $k$. 

\begin{theorem}\label{classtrhas}
The following three sets are in canonical bijection with each other:
\begin{enumerate}
\item
Isomorphism classes of finite dimensional triangular Hopf algebras $(H,R)$ with the Chevalley property over $k$.
\item
Conjugacy classes of quadruples $(\mathcal{G},\mathcal{H},\mathcal{J},\epsilon)$, where $\mathcal{G}$ is a finite supergroup scheme over $k$, $\mathcal{H}\subset \mathcal{G}$
is a closed supergroup subscheme, $\mathcal{J}$ is a minimal twist for $k\mathcal{H}$, and $\epsilon\in\mathcal{G}(k)$ is a central element of order $\le 2$ acting by $-1$ on $\g_1$.
\item 
Conjugacy classes of quadruples $(\mathcal{G},\mathcal{H},\Psi,\epsilon)$, where $\mathcal{G}$ is a finite supergroup scheme over $k$, 
$\mathcal{H}\subset \mathcal{G}$ is a closed supergroup subscheme, $\Psi$ 
is a non-degenerate even $2$-cocycle on $\mathcal{H}$ with coefficients in $\mathbb{G}_m$, and $\epsilon\in\mathcal{G}(k)$ is a central element of order $\le 2$ acting by $-1$ on $\g_1$.
\end{enumerate}
\end{theorem}

\begin{remark}
The correspondence between (1) and (2) in Theorem \ref{classtrhas} is given by $(H,R)=\overline{((k\mathcal{G})^{\mathcal{J}},\epsilon)}$ (see \cite[Theorem 3.3.1]{AEG}; see also \ref{The tensor category sCohfG}). A $2$-cocycle $\Psi$ on $\mathcal{H}$ as in Theorem \ref{classtrhas}(3) determines a module category over $\sRep(\mathcal{G})$ of rank $1$, i.e., a tensor structure on the forgetful functor $\sRep(\mathcal{G})\to \Vect$, thus a twist $\mathcal{J}$ for $k\mathcal{G}$ supported on $\mathcal{H}$.
\end{remark}


\begin{thebibliography}{DGNO}

\bibitem
[AEG]{AEG} N. Andruskiewitch, P. Etingof and S. Gelaki.
Triangular Hopf algebras with the Chevalley property. {\em Michigan Journal of Mathematics} {\bf 49} (2001), 277--298.

\bibitem
[AEGN]{AEGN} E. Aljadeff, P. Etingof, S. Gelaki and D. Nikshych. On twisting of finite-dimensional Hopf algebras. {\em Journal of Algebra} {\bf 256} (2002), 484--501.

\bibitem
[B]{B} R. Bezrukavnikov. On tensor categories attached to cells in affine Weyl groups. Representation theory of algebraic groups and quantum groups, 69--90, {\em Adv. Stud. Pure Math.}, {\bf 40}, Math. Soc. Japan, Tokyo, 2004. 

\bibitem
[D]{DE} P. Deligne. Categories Tannakiennes. In The Grothendick Festschrift, Vol. II, \emph{Prog. Math.} {\bf 87} (1990), 111--195.

\bibitem
[DM]{DM} P. Deligne and J. Milne. Tannakian Categories.
\emph{Lecture Notes in Mathematics}
{\bf 900}, 101--228, 1982.

\bibitem
[EG1]{EG} P. Etingof and S. Gelaki. The classification of triangular semisimple and cosemisimple Hopf algebras over an algebraically closed field. \emph{International Mathematics Research Notices} (2000), no. \textbf{5}, 223--234.

\bibitem
[EG2]{EG3} P. Etingof and S. Gelaki. Finite symmetric integral tensor categories with the Chevalley property, with an Appendix by Kevin Coulembier and Pavel Etingof. {\em International
Mathematics Research Notices}, rnz093, https://doi.org/10.1093/imrn/rnz093.

\bibitem
[EO]{EO} P. Etingof and V. Ostrik. Finite tensor categories.
\emph{Mosc. Math. J.} \textbf{4} (2004), no. 3, 627--654, 782--783.

\bibitem
[EGNO]{egno} P. Etingof, S. Gelaki, D. Nikshych and V. Ostrik. Tensor Categories. {\em AMS Mathematical Surveys and Monographs book series} {\bf 205} (2015), 362 pp.

\bibitem
[G1]{G1} S. Gelaki. On the classification of finite-dimensional triangular Hopf algebras. {\em New directions in Hopf algebras}, 69--116, Math. Sci. Res. Inst. Publ., {\bf 43} (2002), Cambridge Univ. Press, Cambridge.

\bibitem
[G2]{G} S. Gelaki. Module categories over affine group schemes. {\em Quantum Topology} {\bf 6} (2015), no. 1, 1--37.

\bibitem
[Ma]{MA} A. Masuoka. The fundamental correspondences in super affine groups and super formal groups. {\em J. Pure Appl. Algebra} {\bf 202} (2005), 284--312.

\bibitem
[Mo]{Mov} M. Movshev. Twisting in group algebras of finite groups. {\em Func. Anal. Appl.} {\bf 27} (1994), 240--244.

\bibitem
[MS]{MS} A. Masuoka and T. Shibata. On functor points of  affine supergroups.

\bibitem
[R]{R} D.E. Radford. Minimal quasitriangular Hopf algebras. {\em Journal of Algebra} {\bf 157} (1993), 285--315.

\bibitem
[W]{W} W. Waterhouse. Introduction to affine group schemes. Graduate Texts in Mathematics, {\bf 66}. Springer-Verlag, New York-Berlin, 1979. xi+164 pp. 
\end{thebibliography}
\end{document}